\documentclass[reqno]{amsart}
\usepackage[utf8]{inputenc}
\usepackage{amsmath,amssymb,amsthm,amsfonts}
\usepackage[dvipsnames]{xcolor}
\usepackage{mathrsfs}
\usepackage{verbatim}
\usepackage{mathtools}
\usepackage{float}
\usepackage{bbm}
\usepackage{enumerate}
\usepackage[colorlinks=true, citecolor=blue]{hyperref}
\usepackage[ngerman,english]{babel}
\usepackage[notref,notcite,final]{showkeys}

\numberwithin{equation}{section}
\numberwithin{table}{section}
\numberwithin{figure}{section}

\newtheorem{theorem}{Theorem}[section]
\newtheorem{corollary}[theorem]{Corollary}\newtheorem{proposition}[theorem]{Proposition}
\newtheorem{lemma}[theorem]{Lemma}
\theoremstyle{definition}
\newtheorem{definition}[theorem]{Definition}
\newtheorem{remark}[theorem]{Remark}
\newtheorem{example}[theorem]{Example}

\renewcommand{\Re}{\mathop{\textrm{\upshape{Re}}}}

\renewcommand{\phi}{\varphi}
\newcommand{\eps}{\varepsilon}

\newcommand{\forget}[1]{}
\newcommand{\spk}[1]{\langle #1 \rangle}

\newcommand{\R}{\mathbb{R}}
\newcommand{\rz}{\mathbb{R}}
\newcommand{\C}{\mathbb{C}}
\newcommand{\N}{\mathbb{N}}

\renewcommand{\epsilon}{\varepsilon}

\newcommand{\op}{\mathop{\text{\upshape{op}}}\nolimits}
\newcommand{\wt}[1]{\widetilde{#1}}

\newcommand{\dbar}{d\hspace*{-0.08em}\bar{}\hspace*{0.1em}}
\newcommand{\trinorm}[1]{|\hspace*{-1pt}|\hspace*{-1pt}|#1|\hspace*{-1pt}|\hspace*{-1pt}|}

\newcommand{\cz}{{\mathbb C}}

\newcommand{\nz}{{\mathbb N}}

\newcommand{\scrC}{\mathscr{C}}
\newcommand{\scrF}{\mathscr{F}}
\newcommand{\scrK}{\mathscr{K}}
\newcommand{\scrL}{\mathscr{L}}
\newcommand{\scrS}{\mathscr{S}}

\newcommand{\calF}{\mathcal{F}}
\newcommand{\calH}{\mathcal{H}}

\newcommand{\lra}{\longrightarrow}

\newcommand{\wh}{\widehat}

\allowdisplaybreaks

\begin{document}
\title{$R$-boundedness of Poisson operators}

\author{Robert Denk}
\address{Robert Denk, Universit\"at Konstanz, Fachbereich f\"ur Mathematik und Statistik, Konstanz, Germany}
\email{robert.denk@uni-konstanz.de}

\author{Nick Lindemulder}
\address{Nick Lindemulder}
\email{nick.lindemulder@gmail.com}

\author{J{\"o}rg Seiler}
\address{J\"org Seiler, Università di Torino, Dipartimento di Matematica, V. Carlo Alberto 10, 10123 Torino, Italy}
\email{joerg.seiler@unito.it}

\begin{abstract}
We investigate the $R$-boundedness of parameter-dependent families of Poisson operators on the half-space $\R^n_+$ in various scales of function spaces. Applications concern maximal $L_q$-regularity for boundary value problems with dynamic boundary conditions.
\end{abstract}


\date{\today}

\maketitle

\setcounter{tocdepth}{3}
\tableofcontents

\section{Introduction}

Poisson operators arise as solution operators to boundary value problems with inhomogeneous boundary conditions. A simple example is the parameter-dependent heat equation in the half-space $\R^n_+:=\{x\in\R^n: x_n>0\}$,
\begin{equation}
    \label{0-1}
    \begin{alignedat}{4}
        \mu^2 u -\Delta u & = f &\quad& \text{ in }\R^n_+,\\
        \gamma_0 u & = g &&\text{ on }\R^{n-1},
    \end{alignedat}
\end{equation}
where $\gamma_0 u= u|_{\R^{n-1}}$ stands for the boundary trace. If $\mu$ is complex with positive real part (and $f$ and $g$ are suitable) this equation is uniquely solvable and the solution can be written in the form $u=R_\mu f+K_\mu g$. The operator $R_\mu$ is the resolvent of the Dirichlet-Laplacian, $K_\mu$ is an example of a parameter-dependent Poisson operator.

In the classical $L^p$-setting, the most natural solution space for the first equation in \eqref{0-1}
is $H_p^2(\R^n_+)$, the classical second-order Sobolev (or Bessel potential) space. To obtain a solution in this space, one needs to impose that the boundary datum $g$ belongs to $B_{pp}^{2-1/p}(\R^{n-1})$, where $B_{pp}^s$ stands for the Besov space or order $s$. However, there are several applications where the boundary data space has less regularity. For instance, this is the case for stochastic PDEs with boundary noise, where $g$ stands for the derivative of a Brownian motion on $\R^{n-1}$, which is known to belong to some Besov space of negative regularity. Another example is given by the class of boundary value problems with dynamical boundary conditions, which are typically understood as Cauchy problems in the product space $L^p(\R^n_+)\times L^p(\R^{n-1})$.

For considering boundary data of low regularity in \eqref{0-1}, one has to take care of the well-posedness of the boundary trace $\gamma_0$. This trace is continuous as a map from $H_p^s(\R^n)$ to $B_{pp}^{s-1/p}(\R^{n-1})$ if and only if $s>1/p$ (see, e.g., \cite{Johnsen-Sickel08}). However, if we modify the first space to $H_{p,\Delta}^s(\R^n_+) := \{ u\in H_p^s(\R^n_+)\mid \Delta u\in L^p(\R^n_+)\}$, the boundary trace $\gamma_0\colon H_{p,\Delta}^s(\R^n_+)\to B_{pp}^{s-1/p}(\R^{n-1})$ is continuous for all $s\in\R$, so the formulation in \eqref{0-1} makes sense. Surprisingly, the solution operator $K_\mu\colon B_{pp}^{s-1/p}\to H_p^s(\R^n_+)$ is continuous for all $s\in\R$, as was shown by Grubb and Kokholm in \cite{Grubb-Kokholm93}. In the same paper, uniform estimates of $R_\mu$ and $K_\mu$ in parameter-dependent norms were established. However, these estimates are not sufficient for the study of time-dependent problems like
\begin{equation}
    \label{0-2}
    \begin{alignedat}{4}
        \partial_t u (t) -\Delta u (t) & = f(t) &\quad& \text{ in }(0,\infty)\times \R^n_+,\\
        \gamma_0 u(t) & = g (t)  &&\text{ on }(0,\infty)\times \R^{n-1},
    \end{alignedat}
\end{equation}
in the setting of \emph{maximal $L^q$-regularity}, where for $f\in L^q((0,\infty); H^s_p(\R^{n}_+))$ and $g\in L^q((0,\infty); B_{pp}^{s-1/p}(\R^{n-1}))$, we seek a solution in the canonical solution space
\[ u\in L^q((0,\infty); H_p^s(\R^n_+))\cap H_q^1((0,\infty); H_p^{s+2}(\R^n_+)). \]
In this context, a stronger property -- called \emph{$R$-boundedness} -- of the families $R_\mu$ and  $K_\mu$ is needed. In fact, it is known since long that $R$-boundedness is equivalent to maximal $L^q$-regularity for all $q\in (1,+\infty)$ in the sense of well-posedness in $L^q$-$L^p$-Sobolev spaces in time and space. For more information on maximal regularity theory and related subjects we refer the reader to the recent series of text books of Hyt\"onen, van Neerven, Veraar and Weis \cite{Hytoenen-vanNeerven-Veraar-Weis16,Hytoenen-vanNeerven-Veraar-Weis17,Hytoenen-vanNeerven-Veraar-Weis23} and the references therein.

The $R$-boundedness of resolvents like $R_\mu$ above is well-known and can be obtained by general results on normally elliptic boundary value problems. Generally, such resolvents can be decomposed in the sum of a parameter-dependent pseudodifferential operator and a so-called parameter-dependent singular Green operator (the resolvent construction can be performed for example in the framework of Boutet de Monvel's algebra, a pseudodifferential calculus for boundary value problems, introduced in \cite{Boutet71} and then developed further over many decades, see \cite{Grubb-Kokholm93,Grubb96,Schrohe-Schulze} for example). The $R$-boundedness of such operator-families has been shown in very different  settings, see for example \cite[Proposition~3.6]{Hummel21}, \cite[Theorem~3.18]{Denk-Krainer07}, and \cite[Theorem~8]{Denk-Seiler15} for some specific results. Main applications include fluid mechanics and geometric evolution equations  (see, e.g., \cite{Pruess-Simonett16}, \cite{Saito21}, \cite{Shibata20}). Recently, $R$-boundedness of the solution operator was considered for non-local operators in \cite{Abels-Grubb25}.

While the above results and applications are connected with the $R$-boundedness of the resolvent,  the $R$-boundedness of families of Poisson operators has been addressed only very recently; in fact, we are only aware of the paper \cite{Hummel21}. In the present work we obtain new results for the $R$-boundedness of Poisson operators in a variety of function spaces.
%, including anisotropic and weighted Sobolev spaces.
Not only we consider Poisson operators originating from classical boundary value problems as descried above, but also a wider class of Poisson operators related to the analysis of so-called "edge-degenerate" differential operators on manifolds with boundary; see for example \cite{RempelSchulze82,Schulze91} for more details. In both cases parameter-dependent Poisson operators have the form
 $$(K_\mu g)(x)=\int_{\rz^{n-1}} e^{ix'\xi'} k(\xi',\mu,x_n)\wh{g}(\xi')\,d\xi',\qquad x=(x',x_n)\in\rz^n_+,$$
with a specific structure of the so-called symbol-kernel $k$, see Definition \ref{def:poisson-symbols} and Definition \ref{def:weak-poisson-symbols}, respectively; here $\wh g$ is the Fourier transform of $g$. In case of scalar operators (acting on complex-valued functions) the classes  of Poisson operators studied in this paper include those of \cite{Hummel21}; on the other hand, we do not consider Banach space valued Poisson operators as has been done there. However, it seems reasonable to expect that our techniques can be extended to more general settings.

As an application, we discuss maximal $L^q$-regularity for several boundary value problems with dynamic boundary conditions, including a variant of the Cahn--Hilliard equation and a road-field model based on a Kolmogorov–Petrovskii–Pis\-co\-nov equation. In both cases, we consider the space $L^p(\R^{n-1})$ on the boundary. This space also appears as a natural boundary space in applications to bulk-surface reaction-sorption-diffusion systems as studied in \cite{Augner-Bothe21}, \cite{Augner-Bothe24}. In contrast to our examples, the solution in these papers has regularity $H_p^2$ in the interior (bulk)  and is coupled to the surface concentration on the boundary. This allows a maximal regularity approach in standard spaces, similar to the analysis in \cite{Denk-Pruess-Zacher08}.

The paper is organized as follows. In Section~\ref{sec:02}, we provide some background on $R$-boundedness and random sequence spaces. In Section~\ref{sec:03}, we define the first  classes of Poisson operators with and without paprameter and show first boundedness and $R$-boundedness results. Our main results can be found in Sections \ref{sec:04} and \ref{sec:05}, where we consider the $R$-boundedness of parameter-dependent families of Poisson operators in Besov spaces of negative order (Theorems~ and \ref{thm:r-boundedness-negative-s} and \ref{thm:r-boundedness-positiv-s}), in (weak) $L^p$-spaces (Theorem~\ref{thm:weak-main01}), in Sobolev spaces (Section \ref{subsec:04.3}), and in further anisotropic and weighted spaces (Section~\ref{sec:05}). Finally, we apply the results to boundary value problems with dynamic boundary conditions in Section~\ref{sec:06}.

Throughout the paper we will use the following notation. We write $\xi=(\xi',\xi_n)$ for (co-)vectors $\xi\in\rz^n$. For functions $a:\rz^n_\xi\to\cz$, we shall denote by $a(D)$ the associated Fourier-multiplier,
 $$a(D)u=\scrF^{-1}(a\wh{u}).$$
In case $a(\xi)=a(\xi')$ is a function independent of $\xi_n$, we shall use the same notation $a(D')$ for the associated Fourier-multiplier on $\rz^n$ and $\rz^{n-1}$, respectively; it will be clear from the context which operator is intended. In particular, we will consider $a(\xi)=\spk{\xi}^s:=(1+|\xi|^2)^{s/2}$ and $a(\xi')=\spk{\xi'}^s:=(1+|\xi'|^2)^{s/2}$ with $s\in\rz$, yielding the operators $\spk{D}^s$ and $\spk{D'}^s$, respectively.

We denote by $H^s_{p}(\rz^n;X)$, $B^s_{pq}(\rz^n;X)$ and $F^s_{pq}(\rz^n;X)$ $(X$ Banach space, $1\le p,q\le+\infty)$ the standard scales of Bessel, Besov and Triebel-Lizorkin spaces, respectively (for a detailed introduction see Section 14.4 and Section 14.6, respectively, in \cite{Hytoenen-vanNeerven-Veraar-Weis23}). We shall also use the anisotropic spaces
\begin{align}\label{eq:weight}
 H^{s,\sigma}_{p}(\rz^n;X):=\langle D'\rangle^{-\sigma}\big(H^{s}_{p}(\rz^n;X)\big)
\end{align}
with the obvious definition of its norm and, analogously, the spaces $B^{s,\sigma}_{pq}(\rz^n;X)$ and $F^{s,\sigma}_{pq}(\rz^n;X)$.

\textbf{Acknowledgement:} The authors thank Gerd Grubb for useful discussions and her interest in the argument.

%%%%%%%%%%%%%%%%%%%%%%%%%%%%%%%%%%%%%%%%%%%%%%%%%%%
\section{Some background on R-boundedness}\label{sec:02}

In this section we summarize some background material on $R$-boundedness, following \cite{Hytoenen-vanNeerven-Veraar-Weis17}.

A (complex) \emph{Rademacher variable} is a random variable, defined on some probability space $(\Omega,\scrF,\mu)$, which is uniformly distributed over the complex
unit-sphere. We shall denote such a random variable by $\varepsilon$. We also consider sequences
$(\varepsilon_n)=(\varepsilon_n)_{n=1}^{+\infty}$ of independent Rademacher variables which will be called
\emph{Rademacher sequences}.

\begin{definition}
Let $X,Y$ be Banach spaces. A subset $\mathcal{T}$ of $\scrL(X,Y)$ is called \emph{$R$-bounded} provided there exist
$p\in[1,+\infty)$ and a constant $C=C_p\ge 0$ such that
 $$\Big\|\sum_{n=1}^N\varepsilon_j T_jx_j\Big\|_{L^p(\Omega;Y)}
     \le C\,\Big\|\sum_{n=1}^N\varepsilon_j x_j\Big\|_{L^p(\Omega;X)}$$
for every choice of $x_1,\ldots,x_N\in X$, $T_1,\ldots,T_N\in\mathcal{T}$, and $N\in\nz$.
The infimum over all such $C$ is called the $R$-bound of $\mathcal{T}$ and shall be denoted by
$R(\mathcal{T})=R_p(\mathcal{T})$.
\end{definition}

Obviously $R$-boundedness implies boundedness. If $\mathcal{S},\mathcal{T}\subset\scrL(X,Y)$ and $\mathcal{U}\subset\scrL(Y,Z)$  are $R$-bounded then
\begin{align*}
 \mathcal{S}+\mathcal{T}&=\{S+T\mid S\in \mathcal{S},\, T\in\mathcal{T}\}\subset\scrL(X,Y),\\
 \mathcal{UT}&=\{UT\mid R\in \mathcal{U},\, T\in\mathcal{T}\}\subset\scrL(X,Z)
\end{align*}
are also $R$-bounded with bounds $R(\mathcal{S}+\mathcal{T})\le R(\mathcal{S})+\R(\mathcal{T})$ and $R(\mathcal{UT})\le R(\mathcal{U})R(\mathcal{T})$.

A fundamental tool for the study of $R$-boundedness is Mikhlin's multiplier theorem; for a proof see \cite[Theorem 2.15]{Hummel21} and the references given there.

\begin{theorem}[Mikhlin's multiplier theorem]\label{ex:Mikhlin}
Let $X$ be a Banach space and let
$A\subset\scrC^\infty(\rz^{n}\setminus\{0\},\scrL(X))$ be such that
 $$\mathcal{A}:=\left\{|\xi|^{|\alpha|}\partial^{\alpha}_{\xi}a(\xi)\mid
   a\in A,\;|\alpha|\le n,\;\xi\in\rz^n\setminus\{0\}\right\}$$
is an $R$-bounded subset of $\scrL(X)$. If $X$ is a UMD space and has Pisier's property $(\alpha)$ then
$\{a(D)\mid a\in A\}$
is an $R$-bounded subset of $\scrL(H^s_p(\rz^n;X))$, $\scrL(B^s_{pq}(\rz^n;X))$, and
$\scrL(F^s_{pq}(\rz^n;X))$, whenever $s\in\rz$, $1<p<+\infty$, and $1\le q\le+\infty$.
Moreover, in any case, there exists a constant $C$ not depending on $A$ such that
 $$R\big(\{a(D)\mid a\in A\}\big)\le C R(\mathcal{A}).$$
\end{theorem}

We shall apply this theorem in the special case when $X=H$ is a Hilbert space.
Let $\calF(\scrL(H))$ denote the space of all functions $a\in\scrC^\infty(\rz^{n}\setminus\{0\},\scrL(H))$ with finite norm
 $$\|a\|_{\calF(\scrL(H))}
    =\sup_{\xi\in\rz^n,\,|\alpha|\le n}
    |\xi|^{|\alpha|}\|\partial^{\alpha}_{\xi}a(\xi)\|_{\scrL(H)}.$$
Since $R$-boundedness in $\scrL(H)$ coincides with usual boundedness, we obtain:

\begin{corollary}\label{cor:Mikhlin}
Let $A$ be a bounded subset of $\calF(\scrL(H))$. Then $\{a(D)\mid a\in A\}$
is an $R$-bounded subset of $\scrL(H^s_p(\rz^n;H))$, $\scrL(B^s_{pq}(\rz^n;H))$, and
$\scrL(F^s_{pq}(\rz^n,H))$, whenever $s\in\rz$, $1<p<+\infty$, and $1\le q\le+\infty$.
In any case, there exists a constant $C$ not depending on $A$ such that
 $$R\big(\{a(D)\mid a\in A\}\big)\le C \sup_{a\in A}\|a\|_{\calF(\scrL(H))}.$$
\end{corollary}

Given a Rademacher sequence and a sequence $(x_n)$ in an arbitrary Banach space $X$, the so-called \emph{Kahane-Khintchine inequality} holds true: For every choice of $0<p<q<+\infty$ there exists a constant $\kappa=\kappa_{pq}$ such that, for arbitary $N\in\nz$,
 $$\Big\|\sum_{n=1}^N \varepsilon_nx_n\Big\|_{L^q(\Omega;X)}
    \le \kappa\, \Big\|\sum_{n=1}^N \varepsilon_nx_n\Big\|_{L^p(\Omega;X)}.$$
For this reason, if a family of operators has a finite $R$-bound with respect to one $p=p_0\in[1,+\infty)$, it has a finite $R$-bound with repect to every choice of $p\in[1,+\infty)$.

It is convenient to introduce the notion of \emph{random sequence space}:

\begin{definition}
Let $p\ge1$ and $X$ be a Banach space. $\varepsilon^p(X)$ denotes the space of all sequences $(x_n)\subset X$ for which the series $\sum\limits_{n=1}^{+\infty}\varepsilon_nx_n$ converges in $L^p(\Omega;X)$.
We equip $\varepsilon^p(X)$ with the norm
 $$\|(x_n)\|_{\varepsilon^p(X)}:=\Big\|\sum_{n=1}^{+\infty}\varepsilon_nx_n\Big\|_{L^p(\Omega,X)}.$$
\end{definition}

One can show that $\varepsilon^p(X)$ is a Banach space and that
 $$\|(x_n)\|_{\varepsilon^p(X)}=\sup_{N\ge 1}\Big\|\sum_{n=1}^{N}\varepsilon_nx_n\Big\|_{L^p(\Omega,X)}.$$
As an immediate consequence of the Kahane-Khintchine inequality,
$\varepsilon^p(X)\cong\varepsilon^q(X)$ for every choice
of $1\le p,q<+\infty$. For this reason one sets
\begin{equation}\label{eq:varepsilon}
 \varepsilon(X):=\varepsilon^2(X)\qquad (\cong \varepsilon^p(X)\text{ for every $1\le p<+\infty$}).
\end{equation}
Due to the orthogonality of the elements of a Rademacher sequence, clearly $\varepsilon(\cz)=\ell^2(\cz)$, the space of square-summable sequences.
$R$-bounded\-ness can be reformulated in the following way:

\begin{remark}\label{rem:rbounded-bounded}
Let $\mathcal{T}$ be a subset of $\scrL(X,Y)$. Set
 $$\mathcal{T}_\varepsilon
    =\big\{\mathbf{T}=(T_n)_{n=1}^{+\infty}\subset\mathcal{T}\mid
    T_n\not=0\text{ for at most finitely many $n\in\nz$}\big\}.$$
Any $\mathbf{T}\in\mathcal{T}_\varepsilon$ induces a bounded map
$\mathbf{T}:\varepsilon(X)\lra\varepsilon(Y)$ by
 $$(x_n)\mapsto (T_1x_1,T_2x_2,T_3x_3\ldots).$$
Then  $\mathcal{T}$ is $R$-bounded if and only if
$\mathcal{T}_\varepsilon$ is a bounded subset of
$\scrL(\varepsilon(X),\varepsilon(Y))$.
\end{remark}

\begin{theorem}\label{lem:varepsilon-interpolation}
Let $(X_0,X_1)$ be an interpolation couple of Banach spaces. Moreover, let $(\cdot,\cdot)_\theta$ denote an  interpolation space formed with the real or complex interpolation method. Then
 $$(\varepsilon(X_0),\varepsilon(X_1))_\theta\hookrightarrow \varepsilon((X_0,X_1)_\theta).$$
If both $X_0$ and $X_1$ are $K$-convex then the embedding is an isomorphism.
\end{theorem}
\forget{
\begin{proof}
This result is based on a retraction-coretraction argument and can be found in \cite[Proposition 3.16]{Kaip-Saal12}, see also the proof of  in \cite[Proposition 3.7]{Kalton-Kunstmann-Weis06}. Note that in both papers, the result is formulated in terms of  Rademacher sequence spaces which are based on real symmetric Rademacher variables; however, the correspondig norm is equivalent to the norm in $\varepsilon(X)$-spaces by \cite[Proposition~6.1.19]{Hytoenen-vanNeerven-Veraar-Weis17}.
\end{proof}
}

We shall not enter in details about the property of $K$-convexity but only assert that all function and distribution spaces considered later on have this property. As a consequence of the previous theorem, $R$-boundedness is stable under interpolation:

\begin{theorem}\label{thm:R-boundedness-interpolation}
Let $(X_0,X_1)$ and $(Y_0,Y_1)$ be interpolation couples of Banach spaces. Moreover, let $(\cdot,\cdot)_\theta$ denote an interpolation space formed with the real or complex interpolation method. Let $\mathscr{T}\subset\scrL(X_0+X_1,Y_0+Y_1)$ and assume that $\mathscr{T}$ by restriction yields sets $\mathscr{T}_j\subset\scrL(X_j,Y_j)$, $j=0,1$, which are $R$-bounded.
If both $X_0$ and $X_1$ are $K$-convex then
 $$\mathscr{T}_\theta\subset\scrL\big((X_0,X_1)_\theta,(Y_0,Y_1)_\theta\big)$$
$($obtained by restriction$)$ is also $R$-bounded. Moreover,
 $$\mathcal{R}(\mathscr{T}_\theta)\lesssim
   \mathcal{R}(\mathscr{T}_0)^{1-\theta}\mathcal{R}(\mathscr{T}_1)^{\theta}.$$
\end{theorem}

Let us now determine the random sequence spaces in some concrete examples.

\begin{lemma}\label{lem:varepsilon-of-lp}
Let $X$ be a Banach space, $(S,\Sigma,\sigma)$ a measure space, and $1\le p<+\infty$. Then
 $$\varepsilon^p(L^p(S;X))\cong L^p(S;\varepsilon^p(X))$$
isometrically.  In particular, $\varepsilon(L^p(S;X))\cong L^p(S;\varepsilon(X))$ and $\varepsilon(L^p(S))\cong L^p(S;\ell^2(\cz))$ in case $X=\cz$.
\end{lemma}
\begin{proof}
Let $(f_n)$ be a sequence belonging to $\varepsilon^p(L^p(S;X))$. Then
\begin{align*}
 \|(f_n)\|_{\varepsilon^p(L^p(S;X))}^p
     &=\Big\|\sum_{n=1}^{+\infty}\varepsilon_n f_n\Big\|^p_{L^p(\Omega;L^p(S;X))}\\
     &=\int_\Omega\int_S \Big\|\sum_{n=1}^{+\infty}\varepsilon_n(\omega) f_n(s)\Big\|_X^p\,
     d\sigma(s)d\mu(\omega)\\
     &=\int_S \Big\|\sum_{n=1}^{+\infty}\varepsilon_n f_n(s)\Big\|_{L^p(\Omega;X)}^p\,
     d\sigma(s)\\
     &=\int_S \|(f_n(s))\|_{\varepsilon^p(X)}^p\,d\sigma(s)\\
     &=\|(f_n)\|_{L^p(S;\varepsilon^p(X))}^p.
\end{align*}
This shows the first claim. The second then follows from \eqref{eq:varepsilon}.
\end{proof}

\forget{
Let us denote by $B^s_{pq}(\rz^n;X)$ and $F^s_{pq}(\rz^n;X)$ the standard scales of Besov and Triebel-Lizorkin spaces, respectively (for a detailed introduction see Section 14.4 and Section 14.6, respectively, in \cite{Hytoenen-vanNeerven-Veraar-Weis23}). We shall also use the anisotropic spaces $B^{s,\sigma}_{pq}(\rz^n;X)$ and $F^{s,\sigma}_{pq}(\rz^n;X)$ defined by
\begin{align}\label{eq:anisotropic-spaces}
\begin{split}
 B^{s,\sigma}_{pq}(\rz^n;X):=\langle D'\rangle^{-\sigma}\big(B^{s}_{pq}(\rz^n;X)\big),\\
 F^{s,\sigma}_{pq}(\rz^n;X):=\langle D'\rangle^{-\sigma}\big(F^{s}_{pq}(\rz^n;X)\big),
\end{split}
\end{align}
with the obvious definition of its norm, where $\spk{D'}^s$ denotes the Forier multiplier with symbol $m(\xi)=m(\xi',\xi_n)=\spk{\xi'}:=(1+|\xi'|^2)^{1/2}$. Restriction to the half-space $\rz^n_+$ with the standard quotient topology leads to the scale $F^{s,\sigma}_{pq}(\rz^n_+;X)$.
}

\begin{corollary}\label{cor:interpolation}
Let $1<p<+\infty$, $1\le q\le +\infty$, and $s,\sigma\in\rz$. Then
\begin{align*}
 \varepsilon(H^{s,\sigma}_p(\rz^n))&\cong H^{s,\sigma}_p(\rz^n;\ell^2(\cz)),\\
     \varepsilon(B^{s,\sigma}_{pq}(\rz^n))&\cong B^{s,\sigma}_{pq}(\rz^n;\ell^2(\cz)),\\
     \varepsilon(F^{s,\sigma}_{pq}(\rz^n))&\cong F^{s,\sigma}_{pq}(\rz^n;\ell^2(\cz))\qquad (1<q<+\infty).
\end{align*}
Analogous relations hold for the spaces on the half-space $\rz^n_+$.
\end{corollary}
\begin{proof}
The Fourier multiplier $M:=\spk{D}^s\langle D'\rangle^{\sigma}$  induces isomorphisms  $H^{s,\sigma}_p(\rz^n)\to L^p(\rz^n)$ as well as
$H^{s,\sigma}_p(\rz^n;\ell^2(\nz))\to L^p(\rz^n;\ell^2(\nz))$. Therefore, using Lemma \ref{lem:varepsilon-of-lp} with $S=\rz^n$
and $X=\cz$,
\begin{align*}
 (f_n)\in \varepsilon(H^{s,\sigma}_p(\rz^n))
 &\iff (Mf_n)\in \varepsilon(L^p(\rz^n))\\
 &\iff x\mapsto ((Mf_n)(x))\in L^p(\rz^n;\ell^2(\cz))\\
 &\iff x\mapsto (f_n(x))\in H^{s,\sigma}_p(\rz^n;\ell^2(\cz));
\end{align*}
recall that $\varepsilon(\cz)\cong\ell^2(\cz)$. The remaining claims then follow from interpolation, using Lemma \ref{lem:varepsilon-interpolation}.
\end{proof}

%%%%%%%%%%%%%%%%%%%%%%%%%%%%%%%%%%%%%%%%
%%%%%%%%%%%%%%%%%%%%%%%%%%%%%%%%%%%%%%%%
\section{Poisson operators}\label{sec:03}

%%%%%%%%%%%%%%%%%%%%%%%%%%%%%%%%%%%%%%%%%%%%%%%%
\subsection{An example}\label{subsec:example}
In order to motivate the class(es) of operators which are subject of the present paper let us consider the Dirichlet problem for the operator $A:=\Delta-1$ in the half-space, i.e.,
\begin{align*}
 (\mu^2 -A) u(x)&=0, & & x=(x',x_n)\in\rz^{n}_+=\{x\mid x_n>0\},\\
 u(x',0)&=g(x'),& & x'\in\rz^{n-1},
\end{align*}
with a complex parameter $\mu$ satisfying $\Re \mu>0$  (i.e., $\mu=\sqrt{\lambda}$ with $\lambda\in\cz\setminus(-\infty,0])$.

Taking the Fourier transform in $x'$, the problem is equivalent to
 $$\tau(\xi',\mu)^2\, \wh{u}(\xi',x_n)-\partial^2_{x_n}\wh{u}(\xi',x_n)=0,\qquad \wh{u}(\xi',0)=\wh{g}(\xi'),$$
where $\tau(\xi',\mu)=\sqrt{1+|\xi'|^2+\mu^2}$ with the principal branch of the square root on $\cz\setminus(-\infty,0]$. The bounded solution of this initial value problem is given by $\wh{u}(\xi',x_n)=e^{-\tau(\xi',\mu)x_n}\wh{g}(\xi')$, which leads to the solution formula
\begin{equation}\label{eq:example}
 u(x)=(K_\mu g)(x):=\int e^{-ix'\xi'}e^{-\tau(\xi',\mu)x_n}\wh g(\xi')\,\dbar\xi'.
\end{equation}
$K_\mu$ is the prototype of a so-called (strongly) \emph{parameter-dependent Poisson operator}. In general, we shall consider operators $K_\mu=\op(k)(\mu)$ of the form
\begin{equation}\label{eq:poisson-operator}
 (K_\mu g)(x',x_n)=\int e^{-ix'\xi'}k(\xi',\mu;x_n)\wh{g}(\xi')\,\dbar\xi',\qquad x_n>0,
\end{equation}
with so-called \emph{symbol-kernels} $k(\xi',\mu;x_n)$ to be specified below;
the parameter $\mu$ will range in a sub-sector of the complex plane
 $$\Sigma=\big\{0\not=\mu\in\rz^2\cong\cz\mid \alpha<\mathrm{arg}\,\mu<\beta\big\},\qquad 0\le\beta-\alpha\le2\pi;$$
also the case $\Sigma=\emptyset$ is permitted.

%%%%%%%%%%%%%%%%%%%%%%%%%%%%%%%%%%%%%%%%%%%%%%%%%%%%%%%%
\subsection{Parameter-dependent symbol-kernels}

For a systematic analysis of Poisson operators we shall need the following definition.

\begin{definition}
Let $d\in\rz$ and $E$ be a Fréchet space. Then $S^d(\rz^{m}\times\Sigma;E)$ denotes the space of all smooth functions
$a:\rz^{m}\times\Sigma\to E$ such that
 $$\trinorm{D^{\alpha}_{\eta}D^\beta_\mu a(\eta,\mu)}\lesssim \spk{\eta,\mu}^{d-|\alpha|-|\beta|}$$
for every $\alpha\in\nz_0^{m}$, every $\beta\in\nz_0^2$, and every continuous semi-norm on $E$; here $\spk{\eta,\mu}=(1+|\eta|^2+|\mu|^2)^{1/2}$.
In case $\Sigma$ is the empty set, we define $\rz^{m}\times\Sigma:=\rz^{m}$ and $a$ is intended to be a function of the variable $\eta$ only; moreover, in the previous estimate, there are no derivatives with respect to $\mu$ and $\spk{\eta,\mu}$ is replaced by $\spk{\eta}:=\spk{\eta,0}$. This convention will be applied throughout the paper.
\end{definition}

We shall apply the previous definition in both cases $\rz^m=\rz^n$ and $\rz^m=\rz^{n-1}$; in these cases we shall replace $\eta$
by $\xi$ and $\xi'$, respectively.

Taking a countable system of semi-norms $\trinorm{\cdot}_j$ determining the topology of $E$, the space
$S^d(\rz^{m}\times\Sigma;E)$ becomes a Fréchet space with the semi-norms
 $$a\mapsto\sup_{\substack{(\eta,\mu)\in\rz^{m}\times\Sigma,\\|\alpha|+|\beta|+j\le N}}
     \trinorm{D^{\alpha}_{\eta}D^\beta_\mu a(\eta,\mu)}_{j}\spk{\eta,\mu}^{-d+|\alpha|+|\beta|}, \qquad N\in\nz.$$

\begin{definition}\label{def:poisson-symbols}
Let $Z$ be a Banach space and $d\in\rz$. The space $S^d_P(\rz^{n-1}\times\Sigma;Z)$
consists of all smooth functions $k:\rz^{n-1}\times\Sigma\times\rz_+\lra Z$ such that
 $$\wt{k}(\xi',\mu;t):=k(\xi',\mu;\spk{\xi',\mu}^{-1}t)\in
   S^{d}(\rz^{n-1}\times\Sigma;\scrS(\rz_+;Z)),$$
Any such $k$ is called a $(Z$-valued$)$ \emph{Poisson symbol-kernel} of order $d$.
\end{definition}

The canonical bijection
 $$k\mapsto\wt{k} :S^d_P(\rz^{n-1}\times\Sigma;Z)\lra S^{d}(\rz^{n-1}\times\Sigma;\scrS(\rz_+;Z))$$
induces a Fréchet topology on $S^d_P(\rz^{n-1}\times\Sigma;Z)$; it is determined by the system of norms
\begin{align}\label{eq:semi-norm}
 \trinorm{k}_{d,(N)}:=\sup_{\substack{(\eta,\mu,t)\in\rz^{n-1}\times\Sigma\times\rz_+,\\|\alpha'|+|\beta|+\ell+\ell'\le N}}
     \|t^\ell D^{\ell'}_t D^{\alpha'}_{\xi'}D^\beta_\mu \wt{k}(\xi',\mu;t)\|_Z\spk{\eta,\mu}^{-d+|\alpha|+|\beta|}
\end{align}
with $N\in\nz$. Of particular importance will be the case $Z=\scrL(X)$ with a Banach space $X$. In this case, with a Poisson symbol-kernel $k\in S^d_P(\rz^{n-1}\times\Sigma;\scrL(X))$ we associate the parameter-dependent  operator
\begin{equation*}
 K_\mu=\op(k)(\mu):\scrS(\rz^{n-1};X)\to\scrS(\rz^{n}_+;X)
\end{equation*}
as in \eqref{eq:poisson-operator}.
Any such $K_\mu$ is called a $(X$-valued) Poisson operator of order $d$. In case $\Sigma=\emptyset$, we intend $K=K_\mu$ to be an operator that does not depend on the parameter.

\begin{example}
Let us return to the example of Section \ref{subsec:example}. The symbol-kernel in \eqref{eq:example} is $k(\xi',\mu,x_n)=e^{-\tau(\xi',\mu)x_n}$, hence
 $$\wt k(\xi',\mu,t)=\exp\Big(-\frac{\tau(\xi',\mu)}{\spk{\xi',\mu}}t\Big).$$
If we choose the sector $\Sigma=\Sigma_\theta=\{\mu\mid |\mathrm{arg}\,\mu|<\theta\}$ with arbitrary $\theta<\pi/2$, then $\spk{\xi',\mu}/\tau(\xi',\mu)$ and $\tau(\xi',\mu)/\spk{\xi',\mu}$ are scalar symbols of order zero on $\rz^{n-1}\times\Sigma$ and $\Re \tau(\xi',\mu)/\spk{\xi',\mu}\gtrsim 1$. A simple calculation $($see the proof of Lemma \ref{lem:uniform} for a similar situation) then shows that $\wt k\in S^0(\rz^{n-1}\times\Sigma;\scrS(\rz_+))$, hence $k\in S^0_P(\rz^{n-1}\times\Sigma)$.
\end{example}
%%%%%%%%%%%%%%%%%%%%%%%%%%%%%%%%%%%%%%%%
\subsection{Characterizations of Poisson operators}

\begin{theorem}\label{thm:poisson-characterization-01}
Let $Z$ be a Banach space and $d\in\rz$. For a smooth function $k:\rz^{n-1}\times\Sigma\times\rz_+\lra Z$ the following
are equivalent:
\begin{itemize}
 \item[a$)$] $k$ is a Poisson symbol-kernel of order $d$.
 \item[b$)$] For some $1\le p\le+\infty$  and all choices of $\ell,\ell'\in\nz_0$, $\alpha'\in\nz_0^{n-1}$, and $\beta\in\nz_0^2$,
\begin{align*}
   \Big\|x_n\mapsto
       x_n^\ell D_{x_n}^{\ell'} D^{\alpha'}_{\xi'} D^\beta_\mu k(\xi',\mu;x_n)
       \Big\|_{L^p(\rz_+;Z)} \lesssim
       \spk{\xi',\mu}^{d-\frac{1}{p}-\ell+\ell'-|\alpha'|-|\beta|}.
\end{align*}
 \item[c$)$] The statement of b$)$ holds true for every $1\le p\le +\infty$.
\end{itemize}
\end{theorem}
\begin{proof}
Defining $\wt{k}(\xi',\mu;x_n)=k(\xi',\mu;\spk{\xi',\mu}^{-1}x_n)$, by direct calculation it is straightforward to verify that a$)$
implies c$)$ as well as that c$)$ with $p=+\infty$ implies a$)$.  The equivalence of b) and c) has been shown in
\cite[Lemma 4.5]{Hummel-Lindemulder22}.
\end{proof}

A second characterization relates Poisson operators to pseudodifferential operators.
For a precise formulation we need some definitions.

\begin{definition}
A symbol $p(\xi,\mu)\in S^d(\rz^n\times\Sigma;Z)$ is said to have the two-sided transmission property if $k_-,k_+\in S^d(\rz^{n-1}\times\Sigma;\scrS(\rz_+;Z))$, where
 $$k_\pm(\xi',\mu;x_n)
   :=\big[\scrF^{-1}_{\xi_n\to x_n}p(\xi',\spk{\xi',\mu}\xi_n,\mu)\big](\pm x_n).$$
The space of all such symbols is denoted by $S^d_{tr}(\rz^n\times\Sigma;Z)$.
\end{definition}

In more detail, in the previous definition $p(\xi',\spk{\xi',\mu}\xi_n,\mu)$ for each fixed $(\xi',\mu)$ is considered as a tempered
distribution on $\rz$; the transmission property requires that the inverse Fourier transform of this distribution restricted to both
positive and negative half-axis is a regular distribution with density from $\scrS(\rz_+;Z)$ and $\scrS(\rz_-,Z)$, respectively.
Again, the case $\Sigma=\emptyset$ is admitted, meaning that the symbols do not depend on the parameter.

\begin{definition}\label{def:two-sided-trace}
Let $X$ be a Banach space. The two-sided trace operator $\wt\gamma_0:\scrS(\rz^n;X)\to\scrS(\rz^{n-1};X)$ is defined by $(\wt\gamma_0 u)(x')=u(x',0)$.
\end{definition}

%Using the pairing(s) $(T,\varphi)\mapsto T(\varphi)$ between distributions and test functions, we let
Let $\wt\gamma_0^*$ denote the dual operator of $\wt\gamma_0$, i.e.,
 $$(\wt\gamma_0^*T)(\varphi)=T(\wt\gamma_0\varphi),\qquad
     \varphi\in\scrS(\rz^n),\quad T\in\scrS'(\rz^{n-1};X);$$
equivalently, $\wt\gamma_0^*T=T\otimes\delta$ where $\delta$ is the usual delta distribution on $\rz$.
%In particular, if $T$ is a regular distribution with density $u\in\scrS(\rz^{n-1},H)$ then
% $$(\wt\gamma_0^*u)(\varphi)=(u\otimes\delta)(\varphi)=\int u(x')\varphi(x',0)\,dx'.$$

\begin{theorem}\label{thm:poisson-characterization-02}
The space of all $X$-valued Poisson operators $K_\mu$ of order $d$ coincides with the space of all operators of the form
 $$u\mapsto r_+\,p(D,\mu)\wt\gamma_0^*u=r_+\,p(D,\mu)(u\otimes\delta): \scrS(\rz^{n-1};X)\lra \scrS(\rz^n_+;X),$$
with $p\in S^{d-1}_{tr}(\rz^n\times\Sigma;\scrL(X))$ and where
$r_+$ denotes the operator of restriction from $\rz^n$ to $\rz^n_+$.
\end{theorem}
\begin{proof}
The scalar case is known from \cite[Proposition 4.1]{Johnsen96}, for the extension to the operator-valued setting see \cite[Lemma 4.13]{Hummel-Lindemulder22}.
\end{proof}

%%%%%%%%%%%%%%%%%%%%%%%%%%%%%%%%%%%%%%%%
\subsection{A first result on $R$-boundedness of Poisson operators}

The following result on the boundedness of Poisson operators in case $\sigma=0$ has been shown in \cite[Theorem 4.25]{Hummel-Lindemulder22}; it extends the corresponding
scalar-valued result of \cite[Theorem 4.3]{Johnsen96} to the Banach space valued setting. The case $\sigma\in\rz$ follows from the simple fact that the
considered Poisson operators commute with the pseudodifferential operator $\langle D'\rangle^{\sigma}$.

\begin{theorem}\label{thm:poisson-boundedness}
Let $1<p<+\infty$, $1\le q\le +\infty$, and $d,s,\sigma\in\rz$. Let $X$ be a Banach space.
Any $X$-valued Poisson operator $K=\op(k)$ of order $d$ extends to bounded operators
\begin{align*}
 K&:B^{s+\sigma}_{pq}(\rz^{n-1};X)\lra B^{s-d+\frac{1}{p},\sigma}_{pq}(\rz^n_+;X),\\
 K&:F^{s+\sigma}_{pp}(\rz^{n-1};X)\lra F^{s-d+\frac{1}{p},\sigma}_{pq}(\rz^n_+;X).
\end{align*}
Moreover, the following map is bilinear and continuous$:$
 $$(k,u)\mapsto \op(k)u:S^{d}_P(\rz^{n-1};\scrL(X))\times
     B^{s+\sigma}_{pq}(\rz^{n-1};X)\lra
     B^{s-d+\frac{1}{p},\sigma}_{pq}(\rz^n_+;X);$$
the analogous statement holds for the Triebel-Lizorkin spaces.
\end{theorem}

We now present a first result on $R$-boundedness of Poisson operators.
It states that sets which are bounded in the topology of Poisson operators are
$R$-bounded as subsets of bounded operators $($in the sense of Theorem
\ref{thm:poisson-boundedness}).

\begin{theorem}\label{thm:poisson-rbounded-01}
Let $1<p<+\infty$, $1\le q\le +\infty$, and $d,s,\sigma\in\rz$.
Moreover, let $\scrK$ be a bounded subset of $S_P^{d}(\rz^{n-1})$. Then
$\{\op(k)\mid k\in\scrK\}$ is an $R$-bounded subset of both
$\scrL(B^{s+\sigma}_{pq}(\rz^{n-1}), B^{s-d+\frac{1}{p},\sigma}_{pq}(\rz^n_+))$ and
$\scrL(F^{s+\sigma}_{pp}(\rz^{n-1}), F^{s-d+\frac{1}{p},\sigma}_{pq}(\rz^n_+))$ $($where $q<+\infty$ in the latter case$)$.
\end{theorem}
\begin{proof}
We shall apply Remark \ref{rem:rbounded-bounded}. To this end let
$\mathbf{k}:=(k_1,\ldots,k_L,0,0,0,\ldots)$ with arbitrary $L\in\nz$ and
arbitrary $k_j\in\scrK$. Then
 $$\mathbf{k}\in S^{d}_P(\rz^{n-1};\ell^\infty(\cz))
    \subset S^{d}_P(\rz^{n-1};\scrL(\ell^2(\cz)));$$
the latter inclusion is induced by the canonical isometric embedding of
$\ell^\infty(\cz)$ into $\scrL(\ell^2(\cz))$.
Moreover, for every $N\in\nz$,
\begin{align}\label{eq:embedding-01}
 \|\mathbf{k}\|_{d,(N)}\le\max_{j=1}^L \|k_j\|_{d,(N)}
    \le C_N:=\sup_{k\in\scrK}\|k\|_{d,(N)},
\end{align}
where $\|\cdot\|_{d,(N)}$ denotes the norms defining the topology of
Poisson symbols as explained around Definition \ref{def:poisson-symbols}
$($with the obvious choice(s) for the Banach space $Z$ in the definition
of $\|\cdot\|_{d,(N)}$,
namely $Z=\scrL(\ell^2(\cz))$
on the left-hand side and $Z=\cz$ on the right-hand side of \eqref{eq:embedding-01}$)$. Since, by Corollary \ref{cor:interpolation},
\begin{align*}
 \varepsilon(B^{s+\sigma}_{pq}(\rz^{n-1}))
  &\cong B^{s+\sigma}_{pq}(\rz^{n-1};\ell^2(\cz)),\\
 \varepsilon(B^{s-d+\frac{1}{p},\sigma}_{pq}(\rz^n_+))
  &\cong B^{s-d+\frac{1}{p},\sigma}_{pq}(\rz^n_+;\ell^2(\cz)),
\end{align*}
it follows from Theorem \ref{thm:poisson-boundedness} that $\op(\mathbf{k})=(\op(k_1),\ldots,\op(k_L),0,0,0,\ldots)$ is a bounded operator from $\varepsilon(B^{s+\sigma+d-\frac{1}{p}}_{pq}(\rz^{n-1}))$ to
$\varepsilon(B^{s+\sigma}_{pq}(\rz^n_+))$ and
 $$\|\op(\mathbf{k})\|_{\scrL\big(\varepsilon(B^{s+\sigma+d-\frac{1}{p}}_{pq}(\rz^{n-1})),\varepsilon(B^{s,\sigma}_{pq}(\rz^n_+))\big)}\le C$$
with a constant $C$ that neither depends on $L$ nor on the choice of
the $k_j$. This proves the claim for the scale of Besov spaces; the claim for the scale
of Triebel-Lizorkin spaces is verified in the same way.
\end{proof}

%%%%%%%%%%%%%%%%%%%%%%%%%%%%%%%%%%%%%%%%%%%%%%%%%%%
\section{$R$-boundedness of parameter-dependent Poisson operators}\label{sec:04}

Starting out from a parameter-dependent Poisson operator $K_\mu$ it is our aim to obtain the $R$-boundedness of $\{|\mu|^\rho K_\mu\mid \mu\in\Sigma\}$ in $\scrL(X,Y)$ for suitable exponents $\rho$ and spaces $X$ and $Y$. It seems natural to apply Theorem \ref{thm:poisson-rbounded-01} in this context. However,
given a symbol-kernel $k(\xi',\mu;x_n)\in S^d_P(\rz^{n-1}\times\Sigma)$ and setting
 $$k_\mu(\xi';x_n)=k(\xi',\mu;x_n),\qquad \mu\in\Sigma,$$
then $k_\mu\in S^d_P(\rz^{n-1})$ but, in general, there is no order $d'$ and no exponent $\rho$ such that $\{|\mu|^\rho k_\mu\mid \mu\in\Sigma\}$ is a bounded subset of $S^{d'}_P(\rz^{n-1})$. In fact, using Theorem \ref{thm:poisson-characterization-01}, we can only assure that
\begin{align*}
   \big\|x_n\mapsto
       x_n^\ell D_{x_n}^{\ell'} k_\mu(\xi';x_n)
       \big\|_{L^p(\rz_+)} \lesssim
       \spk{\xi',\mu}^{d-\frac{1}{p}-\ell+\ell'},
\end{align*}
which has arbitrarily bad growth in $\mu$ as $|\mu|\to+\infty$, since $-\ell+\ell'$ can become arbitrarily large. For this reason, we need to proceed in a different way.

%%%%%%%%%%%%%%%%%%%%%%%%%%%%%%%%%%%%%%%%%%%%%%%%%%%%%
\subsection{$R$-boundedness in Besov spaces}

\begin{lemma}\label{lem:r-boundedness-pseudo}
Let $p(\xi,\mu)\in S^{d}(\rz^n\times\Sigma)$ with $d\le0$ and let $s\in\rz$, $0\le\theta\le 1$.
Moreover, let $1<p<+\infty$ and $1\le q\le+\infty$.
Then $\big\{\spk{\mu}^{-\theta d} p(D,\mu)\mid\mu\in\Sigma\big\}$
is an $R$-bounded subset of
\begin{align*}
\scrL\big(H^{s}_{p}(\rz^n),H^{s-(1-\theta)d}_{p}(\rz^n)\big).
%\scrL&\big(B^{s}_{pq}(\rz^n),B^{s+1-(1-\theta)d-\frac rp}_{pq}(\rz^n)\big),\\
%\scrL&\big(F^{s}_{pq}(\rz^n),F^{s+1-(1-\theta)d-\frac rp}_{pq}(\rz^n)\big).
\end{align*}
The analogous statement is true if we replace $H^s_p$ by $B^s_{pq}$ or $F^s_{pq}$.
\end{lemma}
\begin{proof}
Write $p(D,\mu)=\spk{D}^{(1-\theta)d}q_\mu(D)$ with $q_\mu(\xi)=p(\xi,\mu)\spk{\xi}^{-(1-\theta)d}$.
Using
 $$\spk{\xi,\mu}^{d-|\beta|}=\spk{\xi,\mu}^{\theta d}\spk{\xi,\mu}^{(1-\theta)d-|\beta|}\le\spk{\mu}^{\theta d}
    \spk{\xi}^{(1-\theta)d-|\beta|}$$
and product rule, it is easy to see that
 $$|D^\alpha_\xi q_\mu(\xi)|\lesssim \spk{\mu}^{\theta d}\spk{\xi}^{-|\alpha|},\qquad
    \alpha\in\nz_0^n.$$
Therefore, $\big\{\spk{\mu}^{-\theta d}q_\mu\mid\mu\in\Sigma\big\}$ is a bounded subset of $S^{0}(\rz^n)$.
The result then follows from Mikhlin's theorem, see Corollary \ref{cor:Mikhlin}.
\end{proof}

In an analogous way one shows the following result:

\begin{remark}\label{lem:r-boundedness-pseudo-02}
Let $p(\xi,\mu)\in S^d(\rz^n\times\Sigma)$ with $d>0$ and $s\in\rz$, $1<p<+\infty$, $1\le q\le+\infty$. Then
$\left\{\spk{\mu}^{-d} p(D,\mu)\mid\mu\in\Sigma\right\}$
is an $R$-bounded subset of $\scrL(H^s_{p}(\rz^n),H^{s-d}_{p}(\rz^n))$, $\scrL(B^s_{pq}(\rz^n),B^{s-d}_{pq}(\rz^n))$,
and $\scrL(F^s_{pq}(\rz^n),F^{s-d}_{pq}(\rz^n))$, respectively.
\end{remark}

Recall the two-sided trace operator $\wt\gamma_0$ from Definition \ref{def:two-sided-trace} and its adjoint.
It is known, see \cite[Proposition 2.6]{Johnsen96} for example, that
$\wt\gamma_0^*$ induces continuous operators
 $$
   B^s_{pq}(\rz^{n-1})\lra B^{s-1+\frac1p}_{pq}(\rz^{n}),\qquad
   F^s_{pp}(\rz^{n-1})\lra F^{s-1+\frac1p}_{pq}(\rz^{n})
 $$
for every $s<0$ and $1\le p,q\le +\infty$ (and $p<+\infty$ in case of the Triebel-Lizorkin spaces).
Combining this with the previous Lemma \ref{lem:r-boundedness-pseudo} we find:

\begin{theorem}\label{thm:r-boundedness-negative-s}
Let $s<0$ and $1<p<+\infty$. Let $p(\xi,\mu)\in S^{d-1}(\rz^n\times\Sigma)$ with $d\le0$ and let $0\le r,\theta\le1$. Then
 $$\big\{\spk{\mu}^{-\theta d+\frac rp} r_+p(D,\mu)\wt\gamma_0^*\mid\mu\in\Sigma\big\}$$
is an $R$-bounded subset of
\begin{equation}\label{eq:intersection}
\begin{split}
 \scrL&\big(B^s_{pq}(\rz^{n-1}),B^{s-(1-\theta)d+\frac{1-r}{p}}_{pq}(\rz^n_+)\big)
   \qquad (1\le q\le+\infty),\\
 \scrL&\big(F^s_{pp}(\rz^{n-1}),F^{s-(1-\theta)d+\frac{1-r}{p}}_{pq}(\rz^n_+)\big)
   \qquad (1< q<+\infty).
\end{split}
\end{equation}
In particular, due to Theorem $\ref{thm:poisson-characterization-02}$, if $K_\mu=\op(k)(\mu)$ is a Poisson operator with symbol-kernel $k\in S^d_P(\rz^{n-1}\times\Sigma)$,  $d\le0$, then $\{\spk{\mu}^{-\theta d+\frac rp}K_\mu\mid\mu\in\Sigma\}$ is an $R$-bounded subset of the spaces in \eqref{eq:intersection}.
\end{theorem}
\begin{proof}
The proof is the same for both scales of spaces; hence let us focus on the Besov spaces. Since $p$ has order $d-1$, by
Lemma \ref{lem:r-boundedness-pseudo} $\big\{\spk{\mu}^{-\theta'(d-1)} p(D,\mu)\mid\mu\in\Sigma\big\}$ is an $R$-bounded subset of $\scrL\big(B^{s-1+\frac1p}_{pq}(\rz^{n}),B^{s-1+\frac1p-(1-\theta')(d-1)}_{pq}(\rz^n)\big)$ for every $0\le\theta'\le 1$.
If $\theta'=\frac{\theta d-\frac rp}{d-1}$ then $-\theta'(d-1)=-\theta d+\frac rp$ and the claim follows.
\end{proof}

Let us now turn to the case of non-negative regularity $s\ge0$.

\forget{
Note that the previous theorem is not true for $s>0$, even not in the case $p=q=2$ when $R$-boundedness coincides with usual boundedness, see the following example; we shall study the case $s\ge 0$ in Section \ref{sec:classical-sobolev}.
The validity of Theorem \ref{thm:r-boundedness-negative-s} in case $s=0$ however will remain an open problem.
}

\begin{example}
Consider the parameter-dependent Poisson operator $K_\mu$ associated with the Dirichlet-Laplacian as described in Section $\ref{subsec:example}$. If we consider only $\mu\in(0,+\infty)$, the symbol-kernel of $K_\mu$ is $k(\xi',\mu,x_n)=e^{-\spk{\xi',\mu}x_n}$; it has order $d=0$. Let us determine the operator-norm of
 $$K_\mu:H^s_2(\rz^{n-1})\lra H^{s+\frac{1-r}{2}}_2(\rz^n_+),\qquad \mu\ge 0,$$
with $0\le s<1/2$ and $0\le r\le1$.
We shall use the fact that
 $$\|v\|^2_{H^{s+(1-r)/2}_2(\rz^n_+)}\cong \int\spk{\xi'}^{2s+1-r}
   \big\|\spk{\xi'}^{-1/2}\wh{v}(\xi',\spk{\xi'}^{-1}x_n)\big\|^2_{H^{s+(1-r)/2}_2(\rz_{+,x_n})}\,d\xi'.$$
% cf. \cite{Schulze}. 
For $v=K_\mu u$ this yields
\begin{align*}
 \|K_\mu u\|^2_{H^{s+(1-r)/2}_2(\rz^n_+)}&\cong \int\spk{\xi'}^{2s-r}
   \big\|e^{-\spk{\xi',\mu}\spk{\xi'}^{-1}x_n}\wh{u}(\xi')\big\|^2_{H^{s+(1-r)/2}_2(\rz_{+,x_n})}\,d\xi'\\
 &=\int \spk{\xi'}^{-r}\big\|e^{-\spk{\xi',\mu}\spk{\xi'}^{-1}x_n}\big\|^2_{H^{s+(1-r)/2}_2(\rz_{+,x_n})}
   \spk{\xi'}^{2s}|\wh{u}(\xi')|^2\,d\xi'.
\end{align*}
It follows that
 $$\|K_\mu\|_{\scrL(H^s_2(\rz^{n-1}),H^{s+(1-r)/2}_2(\rz^n_+))}\cong
   \sup_{\xi'\in\rz^{n-1}}\spk{\xi'}^{-\frac r2}
   \big\|e^{-\spk{\xi',\mu}\spk{\xi'}^{-1}x_n}\big\|_{H^{s+(1-r)/2}_2(\rz_{+,x_n})}.$$
Let us for brevity set $a=a(\xi',\mu):=\spk{\xi',\mu}\spk{\xi'}^{-1}$. It is known, cf. \cite[Lemma 5.2]{DiNezza-Palatucci-Valdinoci}, that the even extension of functions induces a continuous extension operator $H^{t}_2(\rz_+)\to H^{t}_2(\rz)$ for $0\le t<1$. Hence we have
\begin{align*}
  \big\|e^{-ax_n}\big\|_{H^{s+(1-r)/2}_2(\rz_{+,x_n})}\cong \big\|e^{-a|x_n|}\big\|_{H^{s+(1-r)/2}_2(\rz_{x_n})}
\end{align*}
Since the Fourier transform of $e^{-|x_n|}$ coincides with $2\spk{\xi_n}^{-2}$,
\begin{align*}
  \big\|e^{-a|x_n|}\big\|^2_{H^{s+(1-r)/2}_2(\rz_{x_n})}
  &\cong \int \spk{\xi_n}^{2s+1-r}\Big|\frac{\spk{\xi_n/a}^{-2}}{a}\Big|^2\,d\xi_n\\
  &= \frac{1}{a}\int \spk{a\xi_n}^{2s+1-r}\spk{\xi_n}^{-4}\,d\xi_n
  \cong a^{2s-r};
\end{align*}
the latter equivalence is true because
 $$\spk{a\xi_n}^2=1+a^2\xi_n^2\le a^2+a^2\xi_n^2=a^2\spk{\xi_n}^2,
    \qquad \xi_n\in\rz,$$
and
 $$\spk{a\xi_n}^2=1+a^2\xi_n^2\ge a^2,\qquad |\xi_n|\ge 1.$$
Hence
 $$\|K_\mu\|_{\scrL(H^s_2(\rz^{n-1}),H^{s+(1-r)/2}_2(\rz^n_+))}
    \cong\sup_{\xi'}\spk{\xi',\mu}^{s-\frac r2}\spk{\xi'}^{-s}
    =\spk{\mu}^{s-\frac r2}.$$
\end{example}

Based on the previous example we seek a version of Theorem \ref{thm:r-boundedness-negative-s} for $s\ge0$, where the factor $\spk{\mu}^{-\theta d+\frac rp}$ is replaced by $\spk{\mu}^{-s-\theta d+\frac rp}$. It will remain an open problem whether such a theorem is true; however, in case $p=q$, we can demonstrate a slightly weaker version with an ``$\eps$-loss" in the decay in $\mu$:

\begin{theorem}\label{thm:r-boundedness-positiv-s}
Let $2<p<+\infty$, $s\ge0$, $0\le\theta,r\le1$, and $K_\mu=\op(k)(\mu)$ with $k\in S^d_P(\rz^{n-1}\times\Sigma)$ of order $d\le0$. If $\eps>0$ is arbitrary, then
 $$\{\spk{\mu}^{s-\theta d+\frac rp-\eps}K_\mu\mid\mu\in\Sigma\}$$
is an $R$-bounded subset of
 $$\scrL\big(B^s_{pp}(\rz^{n-1}),H^{s-(1-\theta)d+\frac{1-r}{p}}_{p}(\rz^n_+)\big).$$
\end{theorem}
\begin{proof}
Recall from Theorem \ref{thm:poisson-characterization-01} that
 $$\|D^{\ell'}_{x_n}D^{\alpha'}_{\xi'} k(\xi',\mu,\cdot)\|_{L^q(\rz_+)}\lesssim
 \spk{\xi',\mu}^{d-\frac1q+\ell'-|\alpha'|}$$
for any choice of $\ell',\alpha'$.  Therefore,
 $$\|D^{\alpha'}_{\xi'} k(\xi',\mu,\cdot)\|_{H^m_q(\rz_+)}\lesssim \spk{\xi',\mu}^{d-\frac1q+m-|\alpha'|}$$
for every $m\ge 0$ (first for $m\in\nz_0$ and then for real $m$ by complex interpolation). If we identify $H^m_q(\rz_+)\cong\scrL(\cz,H^m_q(\rz_+))$ we thus can consider $k$ as an operator-valued pseudodifferential symbol
\begin{align}\label{eq:symb}
 k\in S^{d-\frac1q+m}(\rz^{n-1}\times\Sigma,\scrL(\cz,H^m_q(\rz_+))),\qquad m\ge0,\;1<q<+\infty.
\end{align}
Now let $0<\eps<\rho:=\frac12-\frac1p$. Moreover, let $t$ denote a real number with
 $$0\le t\le s-(1-\theta)d+\frac{1-r}{p}.$$
Due to \eqref{eq:symb} with $q=2$ and $m=t+\rho\pm\eps$,
\begin{align*}
 \|D^{\alpha'}_{\xi'} k(\xi',\mu,\cdot)\|_{\scrL(\cz,H^{t+\rho\pm\eps}_2(\rz_+))}
  &\lesssim \spk{\xi',\mu}^{d-\frac12+t+\rho\pm\eps-|\alpha'|}
  =\spk{\xi',\mu}^{d-\frac1p+t\pm\eps-|\alpha'|}\\
  &=\spk{\xi',\mu}^{\theta d-\frac rp}\spk{\xi',\mu}^{t+(1-\theta)d-\frac{1-r}{p}}\spk{\xi',\mu}^{-|\alpha'|\pm\eps}\\
  &\lesssim\spk{\mu}^{\theta d-\frac rp}\spk{\xi'}^{t+(1-\theta)d-\frac{1-r}{p}}\spk{\mu}^s\spk{\xi'}^{-|\alpha'|\pm\eps}\spk{\mu}^\eps\\
  &\lesssim \spk{\mu}^{s+\theta d-\frac rp+\eps} \spk{\xi'}^{t+(1-\theta) d-\frac{1-r}{p}\pm\eps-|\alpha'|};
\end{align*}
here we have used $\spk{\xi',\mu}^z\lesssim \spk{\xi'}^z\spk{\mu}^{\max(0,z)}$
for arbitrary $z\in\rz$ and that
 $$\max\Big(0,t+(1-\theta)d-\frac{1-r}{p}\Big)\le\max(0,s)=s.$$
This means that
 $$\spk{\mu}^{-s-\theta d+\frac rp-\eps}k \in S^{t+(1-\theta) d-\frac{1-r}{p}\pm\eps}\big(\rz^{n-1},\scrL(\cz,H^{t+\rho\pm\eps}_2(\rz_+))\big)$$
uniformly in $\mu\in\Sigma$.
Mikhlin's theorem then yields the $R$-boundedness in $\mu\in\Sigma$ of
 $$\spk{\mu}^{-s-\theta d+\frac rp-\eps}K_\mu:H^{s\pm\eps}_p(\rz^{n-1})\lra H^{s-t-(1-\theta)d+\frac{1-r}{p}}_p(\rz^{n-1};H^{t+\rho\pm\eps}_2(\rz_+)).$$
Applying  (1.49) in Theorem 1.8 of \cite{Grubb-Kokholm93} with $\delta'=\rho-\eps$ and $\delta=\rho+\eps$,
 $$(H^{t+\rho-\eps}_2(\rz_+),H^{t+\rho+\eps}_2(\rz_+))_{\frac12,p}\subset H^t_p(\rz_+),$$
and thus Theorem \ref{thm:R-boundedness-interpolation} yields the $R$-boundedness in $\mu\in\Sigma$ of
 $$\spk{\mu}^{-s-\theta d+\frac rp-\eps}K_\mu:B^{s}_{pp}(\rz^{n-1})\lra H^{s-t-(1-\theta)d+\frac{1-r}{p}}_p(\rz^{n-1};H^t_p(\rz_+)). $$
Applying this with $t=0$ and $t=s-(1-\theta)d+\frac{1-r}{p}$ yields the claim,
since
$H^{\sigma}_p(\rz^{n-1};L^p(\rz_+))\cap L^p(\rz^{n-1};H^\sigma_p(\rz_+))= H^\sigma(\rz^n_+)$ for every $\sigma\ge0$.
\end{proof}

%%%%%%%%%%%%%%%%%%%%%%%%%%%%%%%%%%%%%%%%%%%%%%%%%
\subsection{$R$-boundedness in (weak) $L^p$-$L^q$-spaces}

First, we will introduce a wider class of Poisson operators which extends the one introduced in  Definition \ref{def:poisson-symbols}.

\begin{definition}
Let $Z$ be a Banach space. Then $\scrS^0(\rz_+;Z)$ denotes the space of all smooth functions $u:\rz_+\to Z$ such that $\spk{t}^\ell (tD_t)^mu$ is bounded on $\rz_+$ for every choice of $\ell,m\in\nz_0$.
\end{definition}

$\scrS^0(\rz_+;Z)$ is a Fréchet space in an obvious way. If $\omega\in\scrC^\infty(\rz_+)$ has a bounded support and $\omega\equiv1$ on $(0,\delta)$ for some $\delta>0$, then $u\in \scrS^0(\rz_+;Z)$ if and only if $(1-\omega)u\in \scrS(\rz_+;Z)$ and $(t\partial_t)^mu\in L^\infty(\rz_+;Z)$ for every $m\in\nz_0$. Clearly, the space of rapidly decreasing functions $\scrS(\rz_+;Z)$ is continuously
embedded in $\scrS^0(\rz_+;Z)$.

\begin{definition}\label{def:weak-poisson-symbols}
Let $Z$ be a Banach space, and $d\in\rz$. The space $S^d_{P,w}(\rz^{n-1}\times\Sigma;Z)$ consists of all smooth functions $k:\rz^{n-1}\times\Sigma\times\rz_+\to Z$ of the form
 $$k(\xi',\mu;x_n)=\wt{k}(\xi',\mu;\spk{\xi',\mu}x_n),\qquad
     \wt{k}(\xi',\mu;t)\in S^{d}(\rz^{n-1}\times\Sigma;\scrS^0(\rz_+;Z)).$$
\end{definition}

$S^d_{P,w}(\rz^{n-1}\times\Sigma;Z)$ is equipped with the Fréchet topology defined by the norms
 $$\|k\|'_{d,(N)}
    =\sup_{\substack{\xi'\in\rz^{n-1},\,t>0\\\ell+m+|\alpha'|+|\beta|\le N}}
     |(tD_t)^m D^{\alpha'}_{\xi'}D^\beta_\mu\wt{k}(\xi',\mu,t)|\spk{t}^\ell\spk{\xi',\mu}^{-d+|\alpha'|+|\beta|}.$$

As before, in the previous definition also the case $\Sigma=\emptyset$ is allowed, yielding the corresponding class of parameter-independent Poisson symbol-kernels. Obviously,
 $$S^d_{P}(\rz^{n-1}\times\Sigma;Z)\subset S^d_{P,w}(\rz^{n-1}\times\Sigma;Z).$$
Weak Poisson operators (also) occur naturally in the analysis of so-called ``edge-degenerate" differential operators on manifolds with boundary.
%; the interested reader is refered to \cite{Schulze}, for example.

\begin{lemma}\label{lem:uniform}
Let $k\in S^{0}_{P,w}(\rz^{n-1}\times\Sigma;Z)$ and define $k_\mu(\xi',x_n)=k(\xi',\mu,x_n)$. Then $k_\mu\in S^{0}_{P,w}(\rz^{n-1};Z)$ and
for every $N$ there exists a constant $C_N\ge0$ independent of $k$ such that,
for all $\mu\in\Sigma$,
 $$\|k_\mu\|'_{0,(N)}\le C_N \|k\|'_{0,(N)}.$$
\end{lemma}
\begin{proof}
Let $\wt k$ be associated with $k$ as in Definition \ref{def:weak-poisson-symbols}. Then
 $$k_\mu(\xi',x_n)=\wt{k}_\mu(\xi',\spk{\xi'}x_n),\qquad \wt{k}_\mu(\xi',t)
     =\wt{k}\Big(\xi',\mu,\frac{\spk{\xi',\mu}}{\spk{\xi'}}t\Big).$$
Now let $\wt{S}^{d,\nu}$ denote the space of all complex-valued smooth functions $p(\xi',\mu)$ satisfying
 $$|D^{\alpha'}_{\xi'}D^\beta_\mu p(\xi',\mu)|\lesssim
    \spk{\xi'}^{\nu-|\alpha'|}\spk{\xi',\mu}^{d-\nu-|\beta|}$$
for derivatives of arbitrary order. Note that $p\in \wt{S}^{d,\nu}$ implies
$D^{\alpha'}_{\xi'}p\in \wt{S}^{d-|\alpha'|,\nu-|\alpha'|}$ and that
$p_j\in \wt{S}^{d_j,\nu_j}$ implies $p_0p_1\in \wt{S}^{d_0+d_1,\nu_0+\nu_1}$;
moreover, $\spk{\xi',\mu}/\spk{\xi'}\in \wt{S}^{0,-1}$ and
$\spk{\xi'}/\spk{\xi',\mu}\in \wt{S}^{0,1}$.
By induction it is then easy to verify that
 $$(t D_t)^{m}D^{\alpha'}_{\xi'}\wt{k}_\mu(\xi',t)=\sum_{\substack{\alpha''\le\alpha'\\ i+|\alpha''|\le|\alpha'|}}
    p_{\alpha'',i}(\xi',\mu)
    ((tD_t)^{m+i}D^{\alpha''}_{\xi'}\wt{k}) \Big(\xi',\mu,\frac{\spk{\xi',\mu}}{\spk{\xi'}}t\Big),$$
where $p_{\alpha'',i}\in \wt{S}^{|\alpha''|-|\alpha'|,|\alpha''|-|\alpha'|}$ do not depend on $k$.
Since $\spk{\xi',\mu}^{-|\alpha''|}\le \spk{\xi'}^{-|\alpha''|}$ and
$\spk{t\spk{\xi',\mu}/\spk{\xi'}}^{-\ell}\le \spk{t}^{-\ell}$, it follows that
 $$|(t D_t)^{m}D^{\alpha'}_{\xi'}\wt{k}_\mu(\xi',t)|\lesssim \|k\|'_{0,(N)}\spk{\xi'}^{-|\alpha'|}\spk{t}^{-\ell}$$
uniformly in $(\xi',\mu,t)$ whenever $\ell+m+|\alpha'|\le N$.
This yields the claim.
\end{proof}

\forget{
\begin{definition}\label{def:mikhlin}
For a Banach space $Z$ let $\calF(Z)$ denote the space of all functions $a\in\scrC^{n-1}(\rz^{n-1}\setminus\{0\},Z)$
with finite norm
 $$\|a\|_{\calF(Z)}=\sup_{\xi'\not=0,\;|\alpha'|\le n-1}\|\partial^{\alpha'}_{\xi'}a\|_Z|\xi'|^{|\alpha'|}.$$
\end{definition}
}

\begin{proposition}\label{thm:boundedness}
Let $1\le p\le+\infty$, $1<q<+\infty$, and $H$ be a Hilbert space.
If $k\in S^{0}_{P,w}(\rz^{n-1};\scrL(H))$ then
 $$\op(k)\in \scrL(L^q(\rz^{n-1};H),L^p(\rz_+;L^q(\rz^{n-1};H))$$
with operator-norm
 $$\|\op(k)\|_{\scrL(L^q(\rz^{n-1};H),L^p(\rz_+;L^q(\rz^{n-1};H))}\le
     C\big\|\|k(\cdot,x_n)\|_{\calF(\scrL(H))}\big\|_{L^p(\rz_+;x_n)},$$
where $C=C_{pq}\ge0$ is a constant not depending on $k$.
\end{proposition}
\begin{proof}
Since $\op(k)f(x)=(k(D',x_n)f)(x')$ for $x\in\rz^n_+$, an application of Mikhlin's theorem (cf. Theorem \ref{ex:Mikhlin} and Corollary \ref{cor:Mikhlin}) yields
\begin{align*}
 \|\op(k)f\|^p_{L^p(\rz_+;L^q(\rz^{n-1};H))}
  &=\int_0^{+\infty}\|k(D',x_n)f\|_{L^q(\rz^{n-1};H)}^p\,dx_n\\
  &\le C\int_0^{+\infty} \|k(\cdot,x_n)\|_{\calF(\scrL(H))}^p\|f\|_{L^q(\rz^{n-1};H)}^p\,dx_n.
\end{align*}
This yields immediately the claim provided we can show that
$x_n\mapsto\|k(\cdot,x_n)\|_{\calF(\scrL(H))}$ belongs to $L^p(\rz_+)$. To this end let $k(\xi',x_n)=\wt{k}(\xi',\spk{\xi'}x_n)$ as
in Definition \ref{def:weak-poisson-symbols}. It is easy to check that
 $$D^{\alpha'}_{\xi'}k(\xi',x_n)=\sum_{\substack{\alpha''\le\alpha'\\ m+|\alpha''|\le|\alpha'|}} p_{\alpha'',m}(\xi')
    (D^{\alpha''}_{\xi'}(tD_t)^m\wt{k})(\xi',\spk{\xi'}x_n),$$
where $p_{\alpha'',m}\in S^{|\alpha''|-|\alpha'|}(\rz^{n-1})$ do not depend on $k$.
Thus, whenever $|\alpha'|\le n-1$,
\begin{align}\label{eq:estimate01}
  \|D^{\alpha'}_{\xi'}k(\xi',x_n)\|_{\scrL(H)}\lesssim\|k\|'_{0,(n+1)}\spk{\xi'}^{-|\alpha'|}\spk{\spk{\xi'}x_n}^{-2}
\end{align}
uniformly in $\xi'$ and $x_n$, hence
$\|k(\cdot,x_n)\|_{\calF(\scrL(H))}\lesssim\|k\|'_{0,(n+1)}\spk{x_n}^{-2}$.
This clearly yields the claim.
\end{proof}

In the following, we shall use the \emph{weak} $L^p$-space (or \emph{Lorentz space}) on $\R_+=(0,\infty)$ with values in a Banach space $X$, denoted by $L^{p,\infty}(\R_+;X)$. Recall that its norm is given as
 \[ \|f\|_{L^{p,\infty}(\R_+;X)} := \sup_{\alpha>0} \alpha \big[\lambda\big(\{t\in\R_+\mid \|f(t)\|_X >\alpha\}\big)\big]^{1/p},\]
where $\lambda$ stands for the Lebesgue measure in $\R_+$.

\begin{lemma}\label{lem:weak-lp}
    Let $f\colon\R_+\to X$ be measurable with
    \begin{equation}
        \label{A-1}
        \|f(t)\|_X \le C t^{-1/p} \; \text{ for almost all } t\in\R_+.
    \end{equation}
    Then $f\in L^{p,\infty}(\R_+;X)$ with $\|f\|_{L^{p,\infty}(\R_+;X)}\le C$.
\end{lemma}

\begin{proof}
    For $\alpha>0$, let $A_\alpha := \{t\in\R_+\mid \|f(t)\|_X >\alpha\}$. For $t\in A_\alpha$, we have
    \[ C t^{-1/p} \ge \|f(t)\|_X > \alpha\]
    and therefore $t^{1/p} < C/\alpha$ which means $t<C^p\alpha^{-p}$.
    We get $A_\alpha\subset (0, C^p\alpha^{-p})$ and therefore $\lambda(A_\alpha)\le C^p \alpha^{-p}$. This implies \[ \|f\|_{L^{p,\infty}(\R_+;X)} = \sup_{\alpha\in\R_+} \alpha (\lambda(A_\alpha))^{1/p} \le \sup_{\alpha\in\R_+} \alpha C\alpha^{-1} = C.\]
The proof is complete.
\end{proof}

As a consequence of Lemma \ref{lem:weak-lp}, if
$m\colon \R_+\to \scrL(X)$ is measurable with
 \[ \| m(x_n) \|_{\scrL(X)} \le C x_n^{-1/p}, \qquad x_n\in\R_+,\]
then the operator
 \[ M\colon X \to L^{p,\infty}(\R_+;X),\qquad u\mapsto m(\cdot)u\]
is continuous. In fact, by the lemma, $x_n\mapsto m(x_n)u \in L^{p,\infty}(\R_+;X)$
with norm not greater than $C \|u\|_X$ for every $u\in X$.

The following theorem is the main result of this section. Let us remark that the stated result in case $p\ge 2$ and for certain strongly parameter-dependent Poisson operators has been obtained in \cite[Proposition~4.11]{Hummel21}. The proof given here is based on the technique from Theorem \ref{thm:poisson-rbounded-01} and permits to treat the larger class of weakly parameter-dependent Poisson operators as well as to cover the whole range $1<p<+\infty$ in a unified way. Although from \cite[Proposition~4.13]{Hummel21} it was known that \eqref{eq:q-greater-2} is, in general, not true for $p<2$, the theorem reveals that weak $L^p$-spaces can be used instead.

\begin{theorem}\label{thm:weak-main01}
Let $1<p,q<\infty$ and $k\in S^{0}_{P,w}(\rz^{n-1}\times\Sigma)$ then
 $$\scrK:=\big\{\spk{\mu}^{\frac1p}\op(k)(\mu)\mid \mu\in\Sigma\big\}\subset
    \scrL(L^q(\rz^{n-1}),L^{p,\infty}(\rz_+;L^q(\rz^{n-1}))$$
is $R$-bounded. If $p\ge 2$ then
\begin{equation}\label{eq:q-greater-2}
\scrK\subset\scrL(L^q(\rz^{n-1}),L^{p}(\rz_+;L^q(\rz^{n-1}))
\end{equation}
is also $R$ bounded. In either case there exists a constant $C\ge 0$ not
depending on $k$ such that
$\mathcal{R}(\scrK)\le C\|k\|'_{0,(n+1)}$.
\end{theorem}

For the proof of this theorem, we need some preparation.
First, we recall some facts from interpolation theory. If $(X_0,X_1)$ is an interpolation couple of Banach spaces, the real interpolation space $(X_0,X_1)_{\theta,q}$ with $1\le q<\infty$ and $0<\theta<1$ consists of all $x\in X_0+X_1$ with
 $$\|x\|_{\theta,q}:=\Big(\int_0^\infty [t^{-\theta} K(t,x)]^q
    \frac{dt}{t}\Big)^{1/q}<\infty,$$
where $K(t,x)=\inf_{x=x_0+x_1}\|x_0\|_{X_0}+t\|x_1\|_{X_1}$. For $q=\infty$
one considers the norm
 $$\|x\|_{\theta,\infty}:=\sup_{t>0} t^{-\theta} K(t,x).$$

Given a measure space $S$, a measurable function $\omega:S\to(0,\infty)$, and a Banach space $X$, the weighted space $L^p_\omega(S;X)$, $1\le p<\infty$, consists of all measurable functions $f:S\to X$ with finite norm
 $$\|f\|_{L^p_\omega(S,X)}
   =\Big(\int_S \|f(s)\|_X^p\omega(s)\,ds\Big)^{1/p}.$$
If $\omega\equiv1$ we write $L^p(S;X)$, of course. The following result is a combination of Theorem 2.2.10 and C.3.14 in \cite{Hytoenen-vanNeerven-Veraar-Weis16}:

\begin{theorem}\label{thm:interpolation01}
Let $(X_0,X_1)$ be an interpolation couple. Let $1\le p_0,p_1<\infty$, $0<\theta<1$, and $\frac{1}{p}=\frac{1-\theta}{p_0}+\frac{\theta}{p_1}$. Then
 $$\left(L^{p_0}(S;X_0),L^{p_1}(S;X_1)\right)_{\theta,p}
     =L^p(S;(X_0,X_1)_{\theta,p}).$$
Moreover, there exist a constant $C=C(p_0,p_1,\theta)\ge1$ which does not depend on the involved interpolation couple such that, for all $f\in L^p(S,(X_0,X_1)_{\theta,p})$,
 $$\frac1C \|f\|_{L^p(S;(X_0,X_1)_{\theta,p})}\le
     \|f\|_{\left(L^{p_0}(S;X_0),L^{p_1}(S;X_1)\right)_{\theta,p}}\le
     C\|f\|_{L^p(S;(X_0,X_1)_{\theta,p})},$$
\end{theorem}

The following result is Theorem 1.18.5 (and its proof) in \cite{Triebel78}:

\begin{theorem}\label{thm:interpolation02}
Let $X$ be a Banach space. Let $1\le p_0,p_1<\infty$, $0<\theta<1$, and $\frac{1}{p}=\frac{1-\theta}{p_0}+\frac{\theta}{p_1}$. Then
 $$\left(L^{p_0}_{\omega_0}(\Omega;X),L^{p_1}_{\omega_1}
   (\Omega;X)\right)_{\theta,p}=L^p_\omega(\Omega;X),\qquad
    \omega^{\frac{1}{p}}=\omega_0^\frac{1-\theta}{p_0}\omega_1^\frac{\theta}{p_1}.$$
Moreover, there exists a constant $C=C(p_0,p_1,\theta)\ge1$ which does neither depend on the involved weight-functions nor on $X$ such that, for all $f\in L^p_\omega(\Omega;X)$,
 $$\frac1C \|f\|_{L^p_\omega(\Omega;X)}\le
     \|f\|_{\left(L^{p_0}_{\omega_0}(\Omega:X),L^{p_1}_{\omega_1}(\Omega;X)\right)_{\theta,p}}\le
     C\|f\|_{L^p_\omega(\Omega;X)},$$
\end{theorem}

\begin{lemma}\label{lem:max}
Let $\rho>1$, $0<a<\rho$, and $\omega(s,t)=s^a\spk{st}^{-\rho}$. Then
 $$\max_{s\ge 1}\omega(s,t)=
     \begin{cases}
      \spk{t}^{-\rho}&:t\ge (\frac\rho a-1)^{-1/2}\\
      c_{a,\rho}\, t^{-a}&:0<t\le (\frac\rho a-1)^{-1/2}
     \end{cases},\qquad c_{a,\rho}=\Big(1-\frac{a}{\rho}\Big)^{\rho/2}\Big(\frac{\rho}{a}-1\Big)^{-a/2}.$$
\end{lemma}
\begin{proof}
By writing $\omega(s,t)=\big(s^\alpha\spk{st}^{-2}\big)^{\rho/2}$ with $\alpha=\frac{2a}{\rho}$ it
suffices to prove the theorem for the case $\rho=2$. To this end, one calculates
 $$\omega_s(s,t)=\partial_s \omega(s,t)=as^{a-1}\,\frac{1-(\frac2a-1)s^2t^2}{\spk{st}^4}.$$
Since $\omega_s(s,t)\le 0$ whenever $t\ge (\frac2a-1)^{-1/2}$ and $s\ge 1$,
 $$\max_{s\ge 1}\omega(s,t)=\omega(1,t)=\spk{t}^{-2},\qquad t\ge \Big(\frac2a-1\Big)^{1/2}.$$
If $t< (\frac2a-1)^{-1/2}$, then $\omega_s(s,t)=0$ if and only if $s= t^{-1}(\frac2a-1)^{-1/2}$, hence
 $$\max_{s\ge 1}\omega(s,t)=\omega\Big(t^{-1}\Big(\frac2a-1\Big)^{-1/2},t\Big)=c_{a,2}t^{-a}$$
with $c_{a,2}$ as stated.
\end{proof}

\begin{proof}[Proof of Theorem $\ref{thm:weak-main01}$]
Let $\mu_1,\ldots,\mu_M\in\Sigma$ and $M\in\nz$ be arbitrarily chosen. Define
$k_j(\xi',x_n)=k(\xi',\mu_j,x_n)$ and
 $$\mathbf{k}=(k_1,\ldots,k_M,0,0,\ldots),\qquad
 \mathbf{k}^{(p)}=(\spk{\mu_1}^{1/p}k_1,\ldots,\spk{\mu_M}^{1/p}k_M,0,0,\ldots).$$
Then, for $|\alpha'|\le n-1$,
\begin{align*}
 \|D^{\alpha'}_{\xi'}\mathbf{k}^{(p)}(\xi',x_n)\|_{\ell^\infty(\cz)}
 &= \max_{1\le j\le M}\spk{\mu_j}^{1/p}|D^{\alpha'}_{\xi'}k_j(\xi',x_n)|\\
 &\le C\max_{1\le j\le M}\|k_j\|'_{0,(n+1)}\spk{\mu_j}^{1/p}\spk{\xi',\mu_j}^{-|\alpha'|}\spk{\spk{\xi',\mu_j}x_n}^{-2}\\
 &\le C\|k\|'_{0,(n+1)}\spk{\xi'}^{-|\alpha'|}\max_{s\ge 1} s^{1/p}\spk{s x_n}^{-2};
\end{align*}
in fact, the first inequality is proven in the same way as \eqref{eq:estimate01} above, the last inequality  holds due to Lemma
\ref{lem:uniform}. Using that $\ell^\infty(\cz)$ is isometrically embedded in $\scrL(\ell^2(\cz))$ and Lemma \ref{lem:max}, it follows
that
 $$\|\mathbf{k}^{(p)}(\cdot,x_n)\|_{\calF(\scrL(\ell^2(\cz)))}\le C\|k\|'_{0,(n+1)} x_n^{-1/p},\qquad x_n>0,$$
and therefore
\begin{align}\label{eq:est00}
 \|\mathbf{k}^{(p)}(D',x_n)\|_{\scrL(L^q(\rz^{n-1};\ell^2(\cz)))}\le C\|k\|'_{0,(n+1)}x_n^{-1/p},\qquad x_n>0;
\end{align}
note that the constant $C$ does not depend on $M$, $\mu_j$, $x_n$, and $k$.
Hence, as explained after Lemma \ref{lem:weak-lp}, the Poisson operator
\begin{align}\label{eq:est01}
 \op(\mathbf{k}^{(p)}): L^q(\rz^{n-1};\ell^2(\cz)) \lra L^{p,\infty}(\rz_+;L^q(\rz^{n-1};\ell^2(\cz))).
\end{align}
is continuous; its operator-norm can be estimated from above by $C_{pq}\|k\|_{0,(n+1)}$ with a constant $C_{pq}$ which does not depend on $M$, $\mu_j$, and $k$.

Given arbitrary $g_j\in L^q(\rz^{n-1})$, $1\le j\le M$, we have
$\mathbf{g}:=(g_1,\ldots,g_M,0,0,\ldots)\in L^q(\rz^{n-1};\ell^2(\cz))$ and
\begin{align*}
 \Big\|\sum_{j=1}^M &\varepsilon_j \spk{\mu_j}^{1/p}\op(k_j) g_j\Big\|_{L^2(\Omega,L^{p,\infty}(\rz_+;L^q(\rz^{n-1})))}\\
 &=\big\|(\spk{\mu_1}^{1/p}\op(k_1)g_1,\ldots,\spk{\mu_M}^{1/p}\op(k_M)g_M,0,0,\ldots)
 \big\|_{\varepsilon(L^{p,\infty}(\rz_+;L^q(\rz^{n-1})))}\\
 &\cong \big\|\op(\mathbf{k}^{(p)})\mathbf{g}\big\|_{L^{p,\infty}(\rz_+;L^q(\rz^{n-1};\ell^2(\cz)))};
\end{align*}
here we have used that
 $$\varepsilon(L^{p,\infty}(\rz_+;L^q(\rz^{n-1}))
     \cong L^{p,\infty}(\rz_+;\eps(L^q(\rz^{n-1})))
     \cong L^{p,\infty}(\rz_+;L^q(\rz^{n-1};\ell^2(\cz))).$$
By \eqref{eq:est01} and since $\varepsilon(L^q(\rz^{n-1}))\cong L^q(\rz^{n-1};\ell^2(\cz))$,
\begin{align*}
 \big\|\op(\mathbf{k}^{(p)})\mathbf{g}\big\|_{L^p(\rz_+;L^q(\rz^{n-1};\ell^2(\cz)))}
 &\le C\|k\|'_{0,(n+1)}\big\|\mathbf{g}\big\|_{L^q(\rz^n_+;\ell^2(\cz))}\\
 &\cong C\|k\|'_{0,(n+1)}\big\|(g_1,\ldots,g_M,0,0,\ldots)\big\|_{\varepsilon(L^q(\rz^{n-1}))}\\
 &=C\|k\|'_{0,(n+1)}
  \Big\|\sum_{n=1}^M \varepsilon_n g_n\Big\|_{L^2(\Omega;L^q(\rz^{n-1}))}
\end{align*}
with a constant not depending on $M$ and the $\mu_j$. This proves the first claim of the theorem.
For the second claim define weight-functions $\omega_p:\nz\to(0,\infty)$ by
 $$\omega_p(j)=\begin{cases}
    \spk{\mu_j}^{-1/2p}&:1\le j\le M,\\ 1&:\text{otherwise.}
 \end{cases}$$
Then \eqref{eq:est01} is equivalent to the continuity of
\begin{align}\label{eq:est02}
 \op(\mathbf{k}): L^q(\rz^{n-1};\ell^2_{\omega_p}(\cz)) \lra
 L^{p,\infty}(\rz_+;L^q(\rz^{n-1};\ell^2(\cz)))
\end{align}
with operator-norm equal that of $\op(\mathbf{k}^{(p)})$.
For given $p$ choose arbitrary $p_0,p_1\in(1,\infty)$ and $\theta\in(0,1)$ such that $\frac{1}{p}=\frac{1-\theta}{p_0}+\frac{\theta}{p_1}$. Then from \eqref{eq:est02} and Theorem \ref{thm:interpolation01} we find
\begin{align}\label{eq:est03}
 \op(\mathbf{k}): L^q(\rz^{n-1};(\ell^2_{\omega_{p_0}}(\cz),\ell^2_{\omega_{p_1}}(\cz))_{\theta,p}) \lra L^{p}(\rz_+;L^q(\rz^{n-1};\ell^2(\cz))),
\end{align}
where the operator-norm is bounded by $C\|k\|'_{0,(n+1)}$ with a constant that does not depend on $M$, $\mu_j$, and $k$. Moreover, we have used that
 $$(L^{p_0,\infty}(\rz_+;X),L^{p_1,\infty}(\rz_+;X))_{\theta,p}=L^p(\rz_+;X)$$
for arbitrary Banach spaces $X$. In case $p\ge2$, we have
 $$(\ell^2_{\omega_{p_0}}(\cz),\ell^2_{\omega_{p_1}}(\cz))_{\theta,2}\subset
    (\ell^2_{\omega_{p_0}}(\cz),\ell^2_{\omega_{p_1}}(\cz))_{\theta,p}$$
where the corresponding norms satisfy $\|\cdot\|_{\theta,q}\le C\|\cdot\|_{\theta,2}$ with a constant $C=C(\theta,q)$. By Theorem \ref{thm:interpolation02} (with $p_0=p_1=p=2$),
 $$(\ell^2_{\omega_{p_0}}(\cz),\ell^2_{\omega_{p_1}}(\cz))_{\theta,2}=
   \ell^2_{\omega_0^{1-\theta}\omega_1^\theta}(\cz)=
   \ell^2_{\omega_p}(\cz)$$
and
\begin{align*}
 \op(\mathbf{k}): L^q(\rz^{n-1};\ell^2_{\omega_p}(\cz)) \lra L^{q}(\rz_+;L^q(\rz^{n-1};\ell^2(\cz)))
\end{align*}
with operator-norm bounded by $C\|k\|'_{0,(n+1)}$ with a constant that does not depend on $M$, $\mu_j$, and $k$. The latter is equivalent to the continuity of
\begin{align*}
 \op(\mathbf{k}^{(p)}): L^q(\rz^{n-1};\ell^2(\cz)) \lra L^{p}(\rz_+;L^q(\rz^{n-1};\ell^2(\cz)))
\end{align*}
with same operator-norm as $\op(\mathbf{k})$. Then we can repeat the reasoning for the first claim to obtain the second claim of the theorem.
\end{proof}

\forget{
BLOEDSINN
Observe the following slight extension of Theorem $\ref{thm:weak-main01}:$

\begin{theorem}
Let $1<p,q<\infty$ and $k\in S^{0}_{P,w}(\rz^{n-1}\times\Sigma)$.
Then
 $$\scrK=\Big\{\spk{\mu}^{\frac1p}\op(k)(\mu)\mid \mu\in\Sigma\Big\}\subset
    \scrL(L^q(\rz^{n-1}),L^{\wt p,\infty}(\rz_+;L^q(\rz^{n-1}))$$
is $R$-bounded whenever $1<\wt p\le p$.
\end{theorem}
\begin{proof}
Lemma \ref{lem:max} implies that
 $$\max_{s\ge 1}s^a\spk{st}^{-\rho}\lesssim t^{-b}\qquad
    \forall\; 0<a\le b<\rho.$$
Therefore, the factor $x_n^{-1/p}$ in estimate \eqref{eq:est00} in the proof of Theorem \ref{thm:weak-main01} can be substituted by the factor
$x_n^{-1/\wt{p}}$. Then proceed as before with $\wt p$ in place of $p$.
\end{proof}
}

The proof of Theorem $\ref{thm:weak-main01}$ can also be easily modified to obtain the following version with $\epsilon$-loss:

\begin{theorem}\label{rem:weak-version-01}
Let $1<p<2$, $1<q<+\infty$, and $\eps>0$.
If $k\in S^{0}_{P,w}(\rz^{n-1}\times\Sigma)$ then
$$\scrK:=\big\{\spk{\mu}^{\frac1p-\eps}\op(k)(\mu)\mid \mu\in\Sigma\big\}\subset
    \scrL(L^q(\rz^{n-1}),L^p(\rz_+;L^q(\rz^{n-1}))$$
is $R$-bounded. There exists a constant $C\ge 0$ not depending on $k$ such that
 $$\mathcal{R}(\scrK)\le C\|k\|'_{0,(n+1)}.$$
\end{theorem}

In fact, it suffices to observe that, by Lemma \ref{lem:max},
 $$\max_{s\ge 1}s^{\frac1p-\eps}\spk{st}^{-2}\lesssim
    \min\big\{t^{-\frac1p+\eps},\spk{t}^{-2}\big\}\in L^p(\rz_+).$$
Hence if in the definition of $\mathbf{k}^{(p)}$ in the beginning of the proof of Theorem $\ref{thm:weak-main01}$ one replaces $\spk{\mu_j}^{1/p}$ by
$\spk{\mu_j}^{1/p-\eps}$, in \eqref{eq:est01} one can replace the weak $L^p$-space by the usual $L^p$-space, cf. Proposition \ref{thm:boundedness}; then one proceeds as before.

%%%%%%%%%%%%%%%%%%%%%%%%%%%%%%%%%%%%%%%%%%%%%%%%%
\subsection{$R$-boundedness in classical Sobolev spaces}
\label{subsec:04.3}

In this section we derive a result complementary to Theorem \ref{thm:r-boundedness-positiv-s}. We shall employ the usual Bessel potential spaces $H^s_p(\rz_+;X)$ and $H^s_p(\rz^n_+)$ with $s\in\nz_0$. Recall that
 $$H^s_p(\rz^n_+)=H^s_p(\rz_+;L^p(\rz^{n-1}))\cap L^p(\rz_+;H^s_p(\rz^{n-1})),\qquad s\in\nz_0.$$
We then define the weak spaces
 $$H^s_{p,\infty}(\rz^n_+)=H^s_{p,\infty}(\rz_+;L^p(\rz^{n-1}))\cap L^{p,\infty}(\rz_+;H^s_p(\rz^{n-1})),\qquad s\in\nz_0,$$
where $H^s_{p,\infty}(\rz_+;X)$ consists of those functions whose derivatives up to order $s$ belong to $L^{p,\infty}(\rz_+;X)$. Similarly as in \eqref{eq:weight} we define the anisotropic spaces $H^{s,\sigma}_p(\rz^n_+)$ and $H^{s,\sigma}_{p,\infty}(\rz^n_+)$ with $\sigma\in\rz$.

\begin{theorem}\label{thm:weak-05}
Let $1<p,q<\infty$, $t\in\rz$, $s\in\nz_0$, $d\le 0$, and $0\le\theta\le 1$.
If $k\in S^{d}_{P}(\rz^{n-1}\times\Sigma)$ then
\begin{align*}
  \scrK:=&\big\{\spk{\mu}^{-s-\theta d+\frac1p}\op(k)(\mu)\mid \mu\in\Sigma\big\}\\
  \subset& \,\scrL(H^{s+t}_q(\rz^{n-1}),H^s_{p,\infty}(\rz_+;H^{t-(1-\theta)d}_q(\rz^{n-1}))
\end{align*}
is $R$-bounded. If $p\ge 2$ then $\scrK$ is also an $R$-bounded subset of
 $$\scrL(H^{s+t}_q(\rz^{n-1}),H^s_{p}(\rz_+;H^{t-(1-\theta)d}_q(\rz^{n-1})).$$
\end{theorem}
\begin{proof}
We write
 $$\op(k)(\mu)=\op(k)(\mu)\spk{D',\mu}^{-d}\spk{D',\mu}^d=\op(k_0)(\mu)\spk{D',\mu}^d,$$
where $k_0(\xi',\mu,x_n):=\spk{\xi',\mu}^{-d}k(\xi',\mu,x_n)$ belongs to $S^{0}_{P}(\rz^{n-1}\times\Sigma)$.
Since
 $$\spk{\mu}^{-\theta d}\spk{\xi',\mu}^d\spk{\xi'}^{-(1-\theta) d}\in S^0(\rz^{n-1})$$
uniformly in $\mu$,  by Mikhlin's theorem, $\{\spk{\mu}^{-\theta d}\spk{D',\mu}^d\mid \mu\in\Sigma\}$ is an $R$-bounded subset of $\scrL(H^{t}_q(\rz^{n-1}),H^{t-(1-\theta)d}_q(\rz^{n-1}))$. Thus  it remains to verify the claim in case $d=0$ and $\theta=0$. By conjugation with $\spk{D'}^t$ it is also sufficient to consider $t=0$.

By definition of the norm in $H^s_{p,w}(\rz_+;X)$ the result follows from the $R$-boundedness of the sets
 $$\scrK_\ell:=\Big\{\spk{\mu}^{-s-\theta d+\frac1p}D_{x_n}^\ell\op(k)(\mu)\mid \mu\in\Sigma\Big\},\qquad 0\le \ell\le s,$$
as a subset of $\scrL(H^s_q(\rz^{n-1}),L^{p,\infty}(\rz_+;L^q(\rz^{n-1}))$. However, since $k(\xi,\mu,x_n)=\wt k(\xi,\mu,\spk{\xi',\mu}x_n)$, we have
 $$D_{x_n}^\ell\op(k)(\mu)=\op(k_\ell)(\mu)\spk{D',\mu}^\ell,\qquad
    k_{\ell}\in S^{d}_{P}(\rz^{n-1}\times\Sigma).$$
Moreover, the  operator  family $\{\spk{\mu}^{-s}\spk{D',\mu}^\ell\mid\mu\in\Sigma\}$ is an $R$-bounded subset of
$\scrL(H^s_q(\rz^{n-1}),L^q(\rz^{n-1}))$ due to Lemma \ref{lem:r-boundedness-pseudo-02}.
Then the result follows from Theorem~\ref{thm:weak-main01}.
\end{proof}

\begin{corollary}\label{cor:A01}
Let $1<p<\infty$, $t\in\rz$, $s\in\nz_0$, $d\le 0$, and $0\le\theta\le 1$.
If $k\in S^{d}_{P}(\rz^{n-1}\times\Sigma)$ then
 $$\scrK:=\Big\{\spk{\mu}^{-s-\theta d+\frac1p}\op(k)(\mu)\mid \mu\in\Sigma\Big\}\subset
    \scrL(H^{s+t}_p(\rz^{n-1}),H^{s,t-(1-\theta)d}_{p,\infty}(\rz^{n}_+))$$
is $R$-bounded. If $p\ge 2$ then $\scrK$ is also an $R$-bounded subset of
 $$\scrL(H^{s+t}_p(\rz^{n-1}),H^{s,t-(1-\theta)d}_p(\rz^{n}_+)).$$
\end{corollary}
\begin{proof}
Again it is no loss of generality to restrict ourselves to the case $d=\theta=t=0$.

Theorem \ref{thm:weak-05} yields the $R$-boundedness of $\scrK$ in
$\scrL(H^s_p(\rz^{n-1}),L^{p,\infty}(\rz_+;H^s_p(\rz^{n-1}))$ as well as
$\scrL(H^s_p(\rz^{n-1}),H^s_{p,\infty}(\rz_+;L^p(\rz^{n-1}))$.
In order to obtain the claim, it suffices to pass to the intersection of these spaces.
\end{proof}

\forget{
Using interpolation, the previous Corollary \ref{cor:A01} can be extended further. To this end define
the weak Besov spaces
 $$B^s_{p,q,\infty}(\rz^n_+)=(H^{s_0}_{p,\infty}(\rz^n_+),H^{s_1}_{p,\infty}(\rz^n_+))_{\theta,q},\qquad s>0,$$
where $0<s_0<s_1<+\infty$ and $s=(1-\theta)s_0+\theta s_1$, and the weak Triebel-Lizorkin spaces
 $$F^s_{p,q,\infty}(\rz^n_+)=[H^{s_0}_{p_0,\infty}(\rz^n_+),B^{s_1}_{p_1,q_1,\infty}(\rz^n_+)]_{\theta},\qquad s>0,$$
 where $0<s_0<s_1<+\infty$, $s=(1-\theta)s_0+\theta s_1$, $\frac1p=\frac{1-\theta}{p_0}+\frac{\theta}{p_1}$ with $1<p_0,p_1<+\infty$, and $\frac1q=\frac{1-\theta}{2}+\frac{\theta}{p_1}$. Similarly to \ref{eq:weight} we then define the anisotropic spaces $B^{s,\sigma}_{p,q,\infty}(\rz^n_+)$ and $F^{s,\sigma}_{p,q,\infty}(\rz^n_+)$ with $\sigma\in\rz$.

\begin{corollary}\label{cor:A02}
Let $1<p<\infty$, $1\le q\le\infty$, $t\in\rz$, $s\in\N_0$, $d\le 0$, and $0\le\theta\le 1$.
If $k\in S^{d}_{P}(\rz^{n-1}\times\Sigma)$ then
$\big\{\spk{\mu}^{-s-\theta d+\frac1p}\op(k)(\mu)\mid \mu\in\Sigma\big\}$
is an $R$-bounded subset of
\begin{align*}
  \scrL&(H^{s+t}_p(\rz^{n-1}),H^{s,t-(1-\theta)d}_{p,\infty}(\rz^{n}_+)), \\
  \scrL&(B^{s+t}_{pq}(\rz^{n-1}),B^{s,t-(1-\theta)d}_{p,q,\infty}(\rz^{n}_+))\qquad (s>0), \\
  \scrL&(F^{s+t}_{pq}(\rz^{n-1}),F^{s,t-(1-\theta)d}_{p,q,\infty}(\rz^{n}_+))\qquad (s>0,1<q<+\infty).
\end{align*}
In case $p\ge 2$ one can replace the weak spaces with the usual Bessel potential, Besov, and Triebel-Lizorkin spaces,
respectively.
\end{corollary}
}

%%%%%%%%%%%%%%%%%%%%%%%%%%%%%%%%%%%%%%%%%%%%%%%%%
\section{$R$-boundedness in other function spaces}\label{sec:05}

Theorem \ref{thm:weak-main01} together with certain invariance properties of weak Poisson operators allows to obtain
$R$-boundedness in various other scales of function spaces.

%%%%%%%%%%%%%%%%%%%%%%%%%%%%%%%%%%%%%%%%%%%%%%%%%
\subsection{Spaces with totally characteristic normal derivative}

Given a Poisson operator $K_\mu$ with symbol $k\in S^d_{P,w}(\rz^{n-1}\times\Sigma)$ we first observe that $(x_n D_{x_n})^\ell K_\mu$ is a Poisson operator of the same kind; in fact if $k(\xi',\mu,x_n)=\wt{k}(\xi',\mu,\spk{\xi',\mu}x_n)$, cf. Definition \ref{def:weak-poisson-symbols}, then $(x_n D_{x_n})^\ell K_\mu$ has the symbol
 $$k_\ell(\xi',\mu,x_n):=\big((x_n D_{x_n})^\ell\wt{k}\big)(\xi',\mu,\spk{\xi',\mu}x_n).$$

\begin{definition}\label{def:tot-char-01}
Let $X$ be a Banach space and $s\in\nz_0$. Let $\calH^s_p(\rz_+;X)$ denote the space of all functions
$u\in L^p(\rz_+;X)$ such that $(t D_{t})^\ell u\in L^p(\rz_+;X)$ for all $\ell\le s$. In the same way we define $\calH^s_{p,\infty}(\rz_+;X)$ by replacing $L^p(\rz_+;X)$ with the Lorentz space $L^{p,\infty}(\rz_+;X)$.
\end{definition}

Both $\calH^s_{p}(\rz_+;X)$ and $\calH^s_{p,\infty}(\rz_+;X)$ are Banach spaces in an obvious way.

\begin{theorem}\label{thm:weak-02}
Let $1<p,q<\infty$, $t\in\rz$, $s\in\nz_0$, $d\le 0$, and $0\le\theta\le 1$. If $k\in S^{d}_{P,w}(\rz^{n-1}\times\Sigma)$ then
 $$\scrK:=\Big\{\spk{\mu}^{-\theta d+\frac1p}\op(k)(\mu)\mid \mu\in\Sigma\Big\}\subset
    \scrL(H^t_q(\rz^{n-1}),\calH^s_{p,\infty}(\rz_+;H^{t-(1-\theta)d}_q(\rz^{n-1}))$$
is $R$-bounded. If $p\ge 2$ then $\scrK$ is also an $R$-bounded subset of
 $$\scrL(H^t_q(\rz^{n-1}),\calH^s_{p}(\rz_+;H^{t-(1-\theta)d}_q(\rz^{n-1}))).$$
\end{theorem}
\begin{proof}
Since $\op(k)(\mu)=\op(k_0)(\mu)\spk{D',\mu}^d$ where
$k_0(\xi',\mu,x_n):=\spk{\xi',\mu}^dk(\xi',\mu,x_n)$ belongs to $S^{0}_{P,w}(\rz^{n-1}\times\Sigma)$, it suffices to treat the case $d=\theta=0$.  By conjugation with $\spk{D'}^t$ we may also assume $t=0$. Since
 $$\|u\|_{L^p(\Omega,\calH^s_{p,\infty}(\rz_+;L^q(\rz^{n-1})))}\cong
   \sum_{\ell=0}^s \|(tD_t)^\ell u\|_{L^p(\Omega,L^{p,\infty}(\rz_+;L^q(\rz^{n-1})))},$$
the $R$-boundedness of $\{\spk{\mu}^{\frac1p}\op(k)(\mu)\mid \mu\in\Sigma\}$
in $\scrL(L^q(\rz^{n-1}),\calH^s_{p,\infty}(\rz_+;L^q(\rz^{n-1}))$ follows from the $R$-boundedness of
the sets $\{\spk{\mu}^{\frac1p}(x_nD_{x_n})^\ell\op(k)(\mu)\mid \mu\in\Sigma\}$, $\ell\le s$,
in $\scrL(L^q(\rz^{n-1}),L^{p,\infty}(\rz_+;L^q(\rz^{n-1}))$. However, this is true in view of Theorem \ref{thm:weak-main01} and the discussion before Definition \ref{def:tot-char-01}. The reasoning in case $p\ge2$ is the same.
\end{proof}

%%%%%%%%%%%%%%%%%%%%%%%%%%%%%%%%%%%%%%%%%%%%%%%%%
\subsection{Edge degenerate Sobolev spaces}

For a multi-index $\alpha\in\nz_0^n$,  let us write
 $$(x_n D)^\alpha:=(x_n D_{x_1})^{\alpha_1} (x_n D_{x_2})^{\alpha_2} \cdots (x_n D_{x_n})^{\alpha_n}.$$

\begin{definition}\label{def:tot-char-02}
For $s\in\nz_0$ let us define
$\mathcal{W}^{s}_p(\rz^n_+)$ as the space of all functions $u\in L^{p}(\rz^n_+)$ such that
 $$(x_nD)^{\alpha}u\in L^{p}(\rz^n_+)\qquad\forall\;|\alpha|\le s.$$
Furthermore, with $\sigma\in\rz$, we introduce
\begin{align*}
 \mathcal{W}^{s,\sigma}_p(\rz^n_+)
    =\spk{D'}^{-\sigma}\mathcal{W}^{s}_p(\rz^n_+).
\end{align*}
Similarly we define $\mathcal{W}^{s,\sigma}_{p,\infty}(\rz^n_+)$ by substituting
$L^{p}(\rz^n_+)$ with $L^{p,\infty}(\rz_+;L^{p}(\rz^{n-1}))$.
\end{definition}

%The differential operator $(x_n D)^\alpha$ is ``edge-degenerate'' in the sense of \cite{aaa}.
Given a Poisson operator $K_\mu$ with symbol-kernel $k\in S^d_{P,w}(\rz^{n-1}\times\Sigma)$ where
$k(\xi,\mu,x_n)=\wt k(\xi,\mu,\spk{\xi',\mu}x_n)$, we find that $x_n D_{x_j}K_\mu$, $1\le j\le n-1$,
is a Poisson operator with symbol-kernel
 $$k_j(\xi',\mu,x_n):=\frac{\xi_j'}{\spk{\xi',\mu}} (x_n \wt k)(\xi,\mu,\spk{\xi',\mu}x_n)$$
which has order $d$ again; together with the observation made before Definition \ref{def:tot-char-01}, it follows that $(x_n D)^\alpha K_\mu$ is a Poisson operator with symbol-kernel in $S^d_{P,w}(\rz^{n-1}\times\Sigma)$ for arbitrary $\alpha\in\nz_0^n$. Therefore, arguing as in the proof of Theorem \ref{thm:weak-02}, we obtain the following result$:$

\begin{theorem}\label{thm:weak-03}
Let $1<p<\infty$, $s\in\nz_0$, $t\in\rz$, $0\le\theta\le 1$, and
$k\in S^{d}_P(\rz^{n-1}\times\Sigma)$ with $d\le 0$. Then
 $$\scrK:=\big\{\spk{\mu}^{-\theta d+\frac1p}\op(k)(\mu)\mid \mu\in\Sigma\big\}\subset
    \scrL(H^t_p(\rz^{n-1}),\mathcal{W}^{s,t-(1-\theta)d}_{p,\infty}(\rz^n_+))$$
is $R$-bounded. If $p\ge 2$ then $\scrK$ is also an $R$-bounded subset of
 $$\scrL(H^t_p(\rz^{n-1}),\mathcal{W}^{s,t-(1-\theta)d}_{p}(\rz^n_+)).$$
\end{theorem}

%%%%%%%%%%%%%%%%%%%%%%%%%%%%%%%%%%%%%%%%%%%%%%%%%
\subsection{Weighted spaces}

Given a toplogical space $X\subset\mathscr{D}'(\rz^n_+)$ on the half-space $\rz^n_+$, let us denote by $x_n^\gamma X$, $\gamma\in\rz$, the image of $X$ under the map $u\mapsto x_n^\gamma u$, equipped with the canonically induced topology.

Given a symbol-kernel $k(\xi,\mu,x_n)=\wt k(\xi,\mu,\spk{\xi',\mu}x_n)$ of order $d$, obviously
 $$k_\gamma(\xi,\mu,x_n):=x_n^{-\gamma} k(\xi,\mu,x_n)=
   \spk{\xi',\mu}^{\gamma}(x_n^{-\gamma}\wt k)(\xi,\mu,\spk{\xi',\mu}x_n).$$
belongs to $S^{d+\gamma}_{P,w}(\rz^{n-1}\times\Sigma)$ provided
$x_n^{-\gamma}\wt k$ belongs to $S^{d}(\rz^{n-1}\times\Sigma;\scrS^0(\rz_+))$; this is always the case when $\gamma\le0$ because, in this case, multiplication by $x_n^{-\gamma}$ is a continuous endomorphism of $\scrS^0(\rz_+)$. Moreover,
 $$\|\op(k)(\mu)v\|_{x_n^\gamma X}=\|\op(k_\gamma)(\mu)v\|_{X}.$$
Thus we immediately obtain the following versions of Theorems \ref{thm:weak-02} and \ref{thm:weak-03} in weighted spaces$:$

\begin{theorem}\label{thm:weak-02b}
Let $1<p,q<\infty$, $t\in\rz$, $s\in\nz_0$, $\gamma\le 0$, $0\le\theta\le 1$, and
$k\in S^{d}_{P,w}(\rz^{n-1}\times\Sigma)$ with $d+\gamma\le0$. Then
\begin{align*}
 \scrK&:=\Big\{\spk{\mu}^{-\theta(d+\gamma)+\frac1p}\op(k)(\mu)\mid \mu\in\Sigma\Big\}\\
 &\subset \scrL(H^t_q(\rz^{n-1}),x_n^\gamma\calH^s_{p,\infty}(\rz_+;H^{t-(1-\theta)(d+\gamma)}_q(\rz^{n-1}))
\end{align*}
is $R$-bounded. If $p\ge 2$ then $\scrK$ is also an $R$-bounded subset of
 $$\scrL(H^t_q(\rz^{n-1}),x_n^\gamma \calH^s_{p}(\rz_+;H^{t-(1-\theta)d}_q(\rz^{n-1})).$$
\end{theorem}

\begin{theorem}\label{thm:weak-03b}
Let $1<p<\infty$, $s\in\nz_0$, $t\in\rz$, $\gamma\le0$, $0\le\theta\le 1$, and
$k\in S^{d}_P(\rz^{n-1}\times\Sigma)$ with $d+\gamma\le 0$. Then
 $$\scrK:=\Big\{\spk{\mu}^{-\theta (d+\gamma)+\frac1p}\op(k)(\mu)\mid \mu\in\Sigma\Big\}\subset
    \scrL(H^t_p(\rz^{n-1}),x_n^\gamma\mathcal{W}^{s,t-(1-\theta)(d+\gamma)}_{p,\infty}(\rz^n_+))$$
is $R$-bounded. If $p\ge 2$ then $\scrK$ is also an $R$-bounded subset of
 $$\scrL(H^t_p(\rz^{n-1}),x_n^\gamma\mathcal{W}^{s,t-(1-\theta)(d+\gamma)}_{p}(\rz^n_+)).$$
\end{theorem}

%%%%%%%%%%%%%%%%%%%%%%%%%%%%%%%%%%%%%%%%
\section{Application to pde with dynamical boundary conditions}\label{sec:06}

In this section, we use the above results to show maximal $L^q$-regularity for operators related to evolution equations with dynamic boundary conditions.

\subsection{A first example}
We start with the most simple prototype of such equations and consider
\begin{equation}
    \label{5-1}
    \begin{alignedat}{4}
        \partial_t u(t,x) - (\Delta-1) u(t,x) & = f(t,x) &\quad& \text{ in } \rz_+\times \R^n_+,\\
        \partial_t u(t,x)+\partial_\nu u(t,x) & = g(t,x) &&\text{ on } \rz_+\times \R^{n-1},\\
        u(0,x) & =   0 &&\text{ in } \R^n_+.
    \end{alignedat}
\end{equation}
Here, $\partial_\nu = -\frac{\partial}{\partial x_n}$ is the derivative in direction of the outer normal vector. We want to solve \eqref{5-1} for $f\in L^p(\rz_+\times \R^n_+)$ and $g\in L^p(\rz_+\times\R^{n-1})$ with $1<p<+\infty$.

From the first equation, one would expect the solution $u$ to belong to the classical parabolic solution space $L^p(\rz_+; H_p^2(\R^n_+))\cap H_p^1(\rz_+; L^p(\R^n_+))$. In this case, $\partial_\nu u|_{\rz_+\times\R^{n-1}}$ belongs to some (anisotropic) Besov space in $t$ and $x$ of positive order. So the regularity of the boundary data $g$ is not large enough to apply a classical maximal regularity approach as, e.g., in \cite{Denk-Pruess-Zacher08}.

To overcome this difficulty we decouple $u$ and $v:=\gamma_0 u$ and consider \eqref{5-1} as a Cauchy problem of the form $(\partial_t - A)\binom uv = \binom fg$, where $A$ is the unbounded operator in $X:= L^p(\R^n_+)\times L^p(\R^{n-1})$ defined by
\begin{equation}\label{5-2}
\begin{aligned}
   A\binom uv & := \begin{pmatrix}
   \Delta-1 & 0 \\ -\gamma_1 & 0
\end{pmatrix}\binom uv,\\
D(A) & := \left\{ \binom uv\in X: A\binom uv \in X,\, \gamma_0 u = v\right\},
\end{aligned}
\end{equation}
where $\gamma_1:=\gamma_0\partial_\nu$ is the Neumann boundary condition.
Note that, although $u$ is only required to belong to $L^p(\rz^n_+)$, it is meaningful to define the traces $\gamma_0u$ and $\gamma_1 u$: In fact, the definition of $D(A)$ entails that $\Delta u\in L^p(\rz^n_+)$. If we define the spaces
  $$H_{p,\Delta}^s(\R^n_+):= \{ w\in H_p^s(\R^n_+): \Delta w\in L^p(\rz^n_+)\}, \qquad s\in\rz,$$
it is a classical result that $\gamma_0$ and $\gamma_1$ induce, for  every $s\in\rz$, continuous maps from $H_{p,\Delta}^s(\R^n_+)$ to $B^{s-1/p}_{pp}(\rz^{n-1})$ and $B^{s-1-1/p}_{pp}(\rz^{n-1})$, respectively; see for example \cite[Theorem~7]{Seeley66}, \cite[Chapter~2, Section~6.5]{Lions-Magenes72}, and \cite[Section~6.1]{Roitberg96}.

Recall that an operator $A$ is called $R$-sectorial of angle $\theta$ if the resolvent $(A-\lambda)^{-1}$ exists for all
\[ \lambda\in \Sigma_\theta :=
\{\lambda\in\C\setminus\{0\}: |\arg(\lambda)|<\theta\}\]
and the set $\{\lambda (A-\lambda)^{-1}: \lambda\in \Sigma_{\theta'}\}$ is $R$-bounded for all $\theta'<\theta$. It is well known that $R$-sectoriality with angle $\theta>\frac\pi2$ implies maximal $L^q$-regularity.

\begin{theorem}\label{5.1}
The operator $A$ defined in \eqref{5-2} is $R$-sectorial of angle $\pi$. In particular, it has maximal $L^q$-regularity for every $1<q<+\infty$, i.e.,
for any data $\binom fg\in L^q(\rz_+;X)$, there exists a unique solution
\[ \binom uv = \binom u{\gamma_0 u}\in H_q^1(\rz_+; X) \cap L^q(\rz_+; D(A))\]
of \eqref{5-1} depending continuously on $f$ and $g$.
\end{theorem}

\begin{proof}
    We fix $\theta'<\pi$ and consider the resolvent problem
\begin{equation}
    \label{5-3}
    \begin{alignedat}{4}
        (\mu^2+1) u -\Delta u & = f &\quad&\text{ in }\R^n_+,\\
        \mu^2 v + \partial_\nu u & = g&\quad&\text{ on }\R^{n-1},\\
        u & = v &\quad&\text{ on }\R^{n-1}.
    \end{alignedat}
\end{equation}
Here, we have set $\lambda=\mu^2$ with $\mu \in \Sigma:=\Sigma_{\theta'/2}$.
We are looking for the solution $\binom uv = R(\mu)\binom fg$. In a first step, we make the reduction to $f=0$ by writing
\[ \binom uv = \binom{R_1(\mu)f}0 + R(\mu)\binom 0 {g-\gamma_1 R_1(\mu)f},\]
where $u_1 = R_1(\mu)f $ is the unique solution of the first line in \eqref{5-3} with Dirichlet boundary conditions $\gamma_0 u_1=0$. As it is well known that the Dirichlet Laplacian is $R$-sectorial with angle $\pi$, we obtain the $R$-boundedness of
\[ \{\mu^2 R_1(\mu)\mid \mu\in \Sigma\}\subset \scrL (L^p(\R^n_+)).\]
Moreover, $\{ R_1(\mu)\mid \mu\in \Sigma\}$ is an
$R$-bounded subset of $\scrL(L^p(\R^n_+), H_p^2(\R^n_+))$ and $\gamma_1\in \scrL (H_p^2(\R^n_+), L^p(\R^{n-1}))$ is bounded (and independent of $\mu$). Therefore, if we define
 $$\wt R(\mu)g=\begin{pmatrix} \wt R_1(\mu)g\\ \wt R_2(\mu)g\end{pmatrix}
    :=R(\mu)\binom 0 g,\qquad g\in L^p(\R^{n-1}),$$
it remains to show that
\[ \left\{\mu^2\wt R(\mu) \mid \mu\in \Sigma\right\} \subset \scrL (L^p(\R^{n-1}), X)\]
is $R$-bounded. So, we have to solve \eqref{5-3} with $f=0$.
To find an explicit formula for the solution operator,
 we formally  apply  partial Fourier transform in $x'$ and obtain that $u$ has to satisfy
\[ \gamma_1 u = a_\mu(D')\gamma_0 u\]
with symbol $a_\mu(\xi')=\sqrt{1+|\xi'|^2+\mu^2}$. Note that the operator $a_\mu(D')$ is a parameter-dependent version of the Dirichlet-to-Neumann operator.

Inserting this and $\gamma_0 u = v$ into the second line of \eqref{5-3}, we obtain
\begin{equation}\label{5-3a}
 \mu^2\wt R_2(\mu)g=\mu^2 v = b_\mu(D')g,\qquad
 b_\mu(\xi')=\frac{\mu^2}{\mu^2+ \sqrt{1+|\xi'|^2+\mu^2}}.
\end{equation}
It is easily seen that $b_\mu$ satisfies the Mikhlin condition uniformly in $\mu\in\Sigma$, and therefore
\[ \{ b_\mu(D')\mid \mu\in\Sigma\}\subset \scrL(L^p(\R^{n-1}))\] is $R$-bounded. Finally, the first component $u=\wt R_1(\mu)g$ of the solution is given as the unique solution of
\begin{align*}
    (\mu^2+1) u -\Delta u & = 0  \quad \text{ in }\R^n_+,\\
        \gamma_0 u & = v  \quad \text{ on }\R^{n-1}.
\end{align*}
Therefore, we have $u=K_\mu v=K_\mu b_\mu(D')g$ where $K_\mu=\op(k)(\mu)$ is a Poisson operator with symbol $k\in S^{0}_{P}(\rz^{n-1}\times\Sigma)$. Applying Theorem~\ref{thm:weak-main01} in the case $p\ge 2$ and Remark~\ref{rem:weak-version-01} in the case $p<2$, we see that in all cases
\[ \{ \spk{\mu}^{1/p-\eps} \op(k)(\mu)\,|\, \mu\in\Sigma\}\subset
\scrL(L^p(\R^{n-1}), L^p(\R^n_+))\]
is $R$-bounded for any $\eps>0$. In our situation, it is sufficient
to choose $\eps=1/p$, which implies the $R$-boundedness of
$\mu\mapsto\mu^2\wt R_1(\mu)= K_\mu b_\mu(D')$. Altogether, we have seen that
$\{ \mu^2 R(\mu)\mid\mu\in\Sigma\}\subset\scrL(X)$
is $R$-bounded. For smooth right-hand sides $f$ and $g$, the function
$R(\mu)(f,g)^\top$ is, by construction, a classical solution of \eqref{5-3}. By density, it is also a solution  for
general $(f,g)^\top\in X$, where now $\partial_\nu$ has to be understood in the generalized
sense  explained above. Therefore, $R(\mu)$ is in fact  the resolvent of $A$ in $X$, which finishes the proof.
\end{proof}

Note that \eqref{5-3a} implies $v\in H^1_p(\rz^{n-1})$, since $b_\mu(D')$ is, for fixed $\mu$, of order $-1$. From the continuity of the Poisson operator, it follows that, for $p\ge 2$, we have $D(A)\subset H^{1+1/p}_p(\rz^n_+)\times H^1_p(\rz^{n-1})$.

\subsection{Cahn--Hilliard equation with dynamic boundary conditions}\label{Sec5.2}

As a second application, we consider a simplified model of a linearized Cahn--Hilliard equation
with dynamics on the boundary, which was discussed in detail in \cite{Pruess-Racke-Zheng06}.
This fourth-order system  has the form
\begin{equation}\label{5-4}
\begin{alignedat}{4}
  (\partial_t+\Delta^2) u & = f && \text{ in } (0,\infty)\times \R^n_+,\\
  \partial_t u + \partial_\nu u -\Delta' u & = g && \text{ on } (0,\infty)\times\R^{n-1},\\
  \partial_\nu \Delta u & = 0 && \text{ on }(0,\infty)\times\R^{n-1}
\end{alignedat}
\end{equation}
(plus initial condition). Again we decouple $u$ and $v:= \gamma_0 u$ and write \eqref{5-4} as a
 Cauchy problem in the  product space $X:= L^p(\R^n_+)\times L^p(\R^{n-1})$ for the operator $A$
 which now is defined by
\begin{equation}\label{5-5}
\begin{alignedat}{2}
     A  \binom uv & := \begin{pmatrix}
-\Delta^2 & 0 \\
-\partial_\nu & \Delta'
\end{pmatrix} \binom uv , \\
  D(A) & := \left\{ \binom uv\in X: \; A \binom uv\in X,\; \gamma_1 \Delta u  =0,\;
\gamma_0 u =v \right\}.
\end{alignedat}
\end{equation}
Similarly as in \eqref{5-2}, the involved Neumann trace $\gamma_1 (\Delta u)$ in the domain of $A$ is well-defined.
This operator has been analyzed in \cite[Section 5]{Denk-Ploss-Rau-Seiler23}, where it was shown that it is sectorial. With the results of the present paper, we can improve this to $R$-sectoriality and
thus get maximal $L^q$-regularity.

\begin{theorem}
For every $\lambda_0>0$, the operator $A-\lambda_0$ defined in \eqref{5-5} is $R$-sectorial of angle $\pi$ and therefore has maximal $L^q$-regularity for every $q\in (1,\infty)$.
\end{theorem}

\begin{proof}
The proof follows the same lines as the proof of Theorem~\ref{5.1} but in a more complicated
situation. It is shown in \cite[proof of Theorem~5.6]{Denk-Ploss-Rau-Seiler23} that instead of
\eqref{5-3a} the function $v$ is now determined by $\mu^2v=b_\mu(D')g$ with
\[ b_\mu(\xi')=\frac{\mu^2(\tau_1(\xi',\mu)+\tau_2(\xi',\mu))}{(\mu^2+|\xi'|^2)(\tau_1(\xi',\mu)+\tau_2(\xi',\mu))+2 \tau_1(\xi',\mu) \tau_2(\xi',\mu)},\]
where $\tau_{1,2}(\xi',\mu) := \sqrt{|\xi'|^2\pm i\mu}$. Moreover, from \cite[Lemma 5.5]{Denk-Ploss-Rau-Seiler23}
we know that, after replacing $|\xi'|$ by $z$, the resulting function $b_\mu(z)$ is holomorphic and bounded
for $(z,\mu)\in \Sigma_\epsilon\times\Sigma_\theta$ for sufficiently small $\epsilon>0$ and $\theta\in (0,\frac\pi 2)$.
As $-\Delta'$ admits an $R$-bounded $H^\infty$-calculus, the corresponding operator family is $R$-bounded.
Now we can proceed exactly as in the proof of Theorem~\ref{5.1}, applying  Theorem~\ref{thm:weak-main01}  and Remark~\ref{rem:weak-version-01}.
\end{proof}

\forget{
\begin{remark}
In  general, the solution $v$ will have regularity $H_p^2(\R^{n-1})$ in the space variable,
and therefore $u$ (which is determined by $(\mu^2+\Delta^2)u=0,\; \gamma_0u=v$) will only
have space regularity $H_p^{2+1/p}(\R^n_+)$.
In this connection, the question of the definition of the  term $\partial_\nu \Delta u$
 appears. It is known that $\partial_\nu\colon H_p^s(\R^n_+)\to B_{pp}^{s-1-1/p}(\R^{n-1})$ is well defined
 (and continuous) if only if $s>1+1/p$. However, we can use the additional information (from the first line of \eqref{5-4}) that $\Delta^2 u\in L^p(\R^n_+)$. If we define the space
 \[ H_{p,\Delta}^s(\R^n_+) := \{ u\in H_p^s(\R^n_+): \Delta u \in L^p(\R^n_+)\}\]
 with canonical norm, then the Neumann trace $\partial_\nu\colon H_{p,\Delta}^s(\R^n_+)\to B_{pp}^{s-1-1/p}(\R^{n-1})$ is well defined and continuous for all $s\in \R$. This is discussed in detail in, e.g., \cite[Lemma~4.8]{Denk-Ploss-Rau-Seiler23}, where additional references can be found. Applying
 this to the term $\partial_\nu\Delta u$, we see that all terms on the left-hand side of \eqref{5-4} are well defined.
\end{remark}
}

\subsection{A Kolmogorov–Petrovskii–Pisconov model with dynamics on the boundary}

As a last example, we study a field-road model from mathematical biology which was proposed in \cite{Berestycki-Roquejoffre-Rossi13} to model, e.g., the propagation of wolves in regions (fields) with men-made corridors (roads). This system consists of a Kolmogorov–Petrovskii–Pisconov equation in $\R^n_+$ (with $n=2$ for the application) and a dynamic boundary condition on the road $\R^{n-1}$. More precisely, it has the form
\begin{equation}
    \label{5-7}
    \begin{alignedat}{4}
          (\partial_t-d \Delta) u & = f (u)&\;& \text{ in } (0,\infty)\times \R^n_+,\\
  - \gamma_0 u+ (\partial_t - d'\Delta')v + k v   & = 0 && \text{ on } (0,\infty)\times\R^{n-1},\\
 d \partial_\nu   u +\gamma_0 u - k v & = 0 && \text{ on }(0,\infty)\times\R^{n-1}.
    \end{alignedat}
\end{equation}
Here, $d,d',k$ are positive parameters, and the nonlinearity $f(u)$ describes the reproduction process. Again, we write the linearized  system as an evolution equation of the form $(\partial_t-A)\binom uv = \binom fg$ (where in the application $g=0$) in the basic space $X:= L^p(\R^n_+)\times L^p(\R^{n-1})$. Now the operator $A$ is given by
\begin{equation}\label{5-8}
\begin{aligned}
A\binom uv & := \begin{pmatrix}
   d\Delta & 0 \\ \gamma_0 & d'\Delta' - k
\end{pmatrix}\binom uv,\\
D(A) & := \left\{ \binom uv\in X: A\binom uv \in X,\, d\partial_\nu u + \gamma_0 u -k v =0\right\}.
\end{aligned}
\end{equation}
To analyze the operator $A$, we consider the resolvent equation $(\mu^2-A)\binom uv = \binom fg$, where again we can reduce this to the case $f=0$. Applying partial Fourier transform, the first line of \eqref{5-7} turns into an ODE in $(0,\infty)$
of the form
\[ (\mu^2+d|\xi'|^2- d\partial_n^2)\hat u(\xi',x_n)=0\quad \text{ in }(0,\infty).\]
For the only stable solution, we get
\[ \hat u(\xi',x) = \exp\left(-{\sqrt{\frac{\mu^2}d+|\xi'|^2}}\;x_n\right)\hat u (\xi',0).\]
Inserting this into the boundary conditions yields
\[ \begin{pmatrix}
-1 & \mu^2+k+d'|\xi'|^2 \\
    \sqrt{d\mu^2+d^2|\xi'|^2}+1 & -k
\end{pmatrix}\binom{\hat u(\xi',0)}{\hat v(\xi')} = \binom{\hat g (\xi')}0.\]
Therefore $\mu^2\gamma_0 u=\op'[m_1(\cdot,\mu)] g$ and $\mu^2v = \op'[m_2(\cdot,\mu)] g$ with
\begin{equation}\label{5-9}
    \begin{aligned}
    m_1(\xi',\mu) & = \frac{ \mu^2 k}{(\mu^2+k+d'|\xi'|^2)(\sqrt{d\mu^2+d^2|\xi'|^2}+1)-k }\;,\\[0.2em]
    m_2(\xi',\mu) & =  \frac{ \mu^2(\sqrt{d\mu^2+d^2|\xi'|^2}+1 )  }{(\mu^2+k+d'|\xi'|^2)(\sqrt{d\mu^2+d^2|\xi'|^2}+1)-k }\;.
\end{aligned}
\end{equation}

As in Subsection~\ref{Sec5.2}, we will use the $H^\infty$-calculus for the operator $-\Delta'$. For this, the
following result is the key observation.

\begin{lemma}
    \label{5.4} Let $\theta\in (0,\pi/2)$, and let $z\in \Sigma_\epsilon$ with  $\epsilon>0$.
    Define $m_1(z,\mu)$ and $m_2(z,\mu)$ by \eqref{5-9} with $|\xi'|$ being replaced by $z$. Then, if $\epsilon$ is chosen small enough, the functions $m_1$ and $m_2$ are bounded and holomorphic for $(z,\mu)\in \Sigma_\epsilon\times \Sigma_\theta$.
\end{lemma}

\begin{proof}
    We only consider $m_1$, as the proof for $m_2$ can be done in the same way. Let $\epsilon\in (0,\pi/2-\theta)$. We first remark that, by
    homogeneity and compactness, for all $(z,\mu)\in (\overline{\Sigma_\epsilon}\times \overline{\Sigma_\theta})
    \setminus\{0\}$ we have
    \begin{equation}\label{5-10}
    c_1 |(z,\mu)| \le \big| \mu^2+z^2\big|^{1/2} \le C_1 |(z,\mu) |,\end{equation}
    where $|(z,\mu)| := := (|z|^2+|\mu|^2)^{1/2} $. From this we obtain
$      m_1(z,\mu)\to 0$ for $|(z,\mu)|\to\infty$. In particular, $m_1$ is bounded for $|(z,\mu)|\ge R$ with $R$ sufficiently large.

For $|(z,\mu)|\to 0$, we write the denominator of $m_1$ in the form
\[ m_1(z,\mu) = \frac{\mu^2k}{(\mu^2+d'z^2)(\sqrt{d\mu^2+d^2z^2} +1) + k \sqrt{d\mu^2+d^2z^2}}\,.\]
Again using \eqref{5-10}, we see that the absolute value of the denominator can be estimated from below by
\[ C \big( |(z,\mu)| - |(z,\mu)|^2 (|(z,\mu)| +1)\big) \ge C'|(z,\mu)|\]
for sufficiently small $|(z,\mu)|$. This yields $m_1(z,\mu)\to 0$ for $|(z,\mu)|\to 0$, and therefore $m_1$ can be extended to a continuous function on $\overline{\Sigma_\epsilon}\times \overline{\Sigma_\theta}$.

     By continuity and compactness, we know that $m_1(z,\mu)$ is bounded for $|(z,\mu)|\le R$ if we can
     show that the denominator in \eqref{5-9} does not vanish. We set
     \[ f(z,\mu) := (\mu^2+k+d'z^2)\big(\sqrt{d\mu^2+d^2z^2} +1\big)\]
     and have to show that $f(z,\mu)\not=k$. For this, we first assume that $z$ is real. As in this case $f(z,\bar \mu)=\overline{f(z,\mu)}$, we may assume $\arg(\mu)\in [0,\theta)$. If also $\mu$ is real, we
     obviously have $f(z,\mu)>k$. If $\arg(\mu)\in (0,\theta)$, we see that $\arg(\mu^2+k+d'z^2)\in (0,2\theta)$
     and $\arg(\sqrt{d\mu^2+d^2z^2}+1)\in (0,\theta)$. Therefore $\arg(f(z,\mu))\in (0, 3\theta)$
     which implies $f(z,\mu)\not\in (0,\infty)$. So we have seen that $f(z,\mu)\not=k$ for all
     $(z,\mu)\in (0,\infty)\times \Sigma_\theta$.

     For $\alpha\in (-\epsilon,\epsilon)$, we define
     \[ g(\alpha) := \min\Big\{ |f(r e^{i\alpha},\mu)-k|: (r,\mu)\in [0,\infty)\times \overline{\Sigma_\theta},\, r^2+|\mu|^2\le R^2\Big\}. \]
We have just seen that $g(0)>0$. As the function $g$ depends (as the minimum of a continuous function over a compact set) continuously on $\alpha$, there exists some $\epsilon'>0$ such that $g(\alpha)>0$ for all $\alpha\in (-\epsilon',\epsilon')$, which finishes the proof of the statement.
\end{proof}

\begin{theorem}
The operator $A$  related to the Kolmogorov–Petrovskii–Pisconov model and defined in \eqref{5-8} is $R$-sectorial of angle $\pi$ and therefore has maximal $L^q$-regularity for every $q\in (1,\infty)$.
\end{theorem}

\begin{proof}
The proof follows the same steps as in the above examples. We reduce to the case $f=0$ and write the solution operator to $(\mu^2-A)\binom uv = \binom 0g$ in the form $\gamma_0 u = R_1(\mu)g$ and $v = R_2(\mu)g$. By Lemma~\ref{5.4} and the fact that $-\Delta'$ has an $R$-bounded $H^\infty$-calculus, we know that $\{\mu^2 R_j(\mu): \mu\in \Sigma_\theta\}$ is $R$-bounded for all $\theta<\frac\pi2$. Now we can use the $R$-boundedness of the Poisson operator related to the boundary value problem  $(\mu^2-d\Delta,\gamma_0)$.
\end{proof}

%%%%%%%%%%%%%%%%%%%%%%%%%%%%%%%%%%%%%%%%%%%%%%%%%%%
\bibliographystyle{alpha}

\end{document}